\documentclass[3p, 12pt]{elsarticle}

\usepackage{hyperref}

\usepackage{amssymb}
\usepackage{appendix}
\usepackage[USenglish]{babel} 

\usepackage{hyperref}

\usepackage{amsfonts,amsmath,latexsym}
\usepackage[active]{srcltx}
\usepackage{enumitem}

\usepackage{verbatim}

\newtheorem{theorem}{Theorem}[section]
\newtheorem{problem}[theorem]{Problem}
\newtheorem{remark}[theorem]{Remark}
\newproof{proof}{Proof}

\newcommand{\RR}{{\if mm {\rm I}\mkern -3mu{\rm R}\else \leavevmode
\hbox{I}\kern -.17em\hbox{R} \fi}}

\newcommand{\bu}{\mbox{\boldmath{$u$}}}
\newcommand{\bt}{\mbox{\boldmath{$t$}}}
\newcommand{\bcero}{\mbox{\boldmath{$0$}}}

\newcommand{\bx}{\mbox{\boldmath{$\bx$}}}
\newcommand{\by}{\mbox{\boldmath{$\by$}}}
\newcommand{\bv}{\mbox{\boldmath{$v$}}}

\newcommand{\var}{\varepsilon}

\newcommand{\ba}{\mbox{\boldmath{$a$}}}

\newcommand{\fb}{{{f}}}
\newcommand{\bg}{\mbox{\boldmath{$g$}}}
\renewcommand{\by}{\mbox{\boldmath{$y$}}}
\renewcommand{\bx}{\mbox{\boldmath{$x$}}}
\newcommand{\be}{\mbox{\boldmath{$e$}}}
\newcommand{\bn}{\mbox{\boldmath{$n$}}}
\newcommand{\bh}{{{h}}}
\newcommand{\bbf}{\mbox{\boldmath{$f$}}}
\newcommand{\bbh}{\mbox{\boldmath{$h$}}}

\newcommand{\btheta}{\mbox{\boldmath{$\theta$}}}

\newcommand{\beeta}{\mbox{\boldmath{$\eta$}}}
\newcommand{\bxi}{\mbox{\boldmath{$\xi$}}}

\newcommand{\bTheta}{\mbox{\boldmath{$\Theta$}}}
\newcommand{\bUcal}{\mbox{\boldmath{$\mathcal{U}$}}}

\newcommand{\Gamae}{\Gamma^{\varepsilon }}

\newcommand{\dep}{(t,\cdot)}

\renewcommand{\d}{\partial}
\newcommand{\eij}{e_{i||j}}
\newcommand{\deij}{\dot{e}_{i||j}}
\newcommand{\ekl}{e_{k||l}}
\newcommand{\dekl}{\dot{e}_{k||l}}
\newcommand{\eab}{e_{\alpha||\beta}}
\newcommand{\est}{e_{\sigma||\tau}}
\newcommand{\dest}{\dot{e}_{\sigma||\tau}}
\newcommand{\estres}{e_{\sigma||3}}
\newcommand{\destres}{\dot{e}_{\sigma||3}}
\newcommand{\eatres}{e_{\alpha||3}}
\newcommand{\edtres}{e_{3||3}}
\newcommand{\deab}{\dot{e}_{\alpha||\beta}}
\newcommand{\deatres}{\dot{e}_{\alpha||3}}
\newcommand{\dedtres}{\dot{e}_{3||3}}
\newcommand{\gab}{\gamma_{\alpha\beta}}

\newcommand{\rab}{\rho_{\alpha\beta}}
\newcommand{\rst}{\rho_{\sigma\tau}}
\renewcommand{\a}{a^{\alpha\beta\sigma\tau}}
\renewcommand{\b}{b^{\alpha\beta\sigma\tau}}
\renewcommand{\c}{c^{\alpha\beta\sigma\tau}}
\newcommand{\aeps}{a^{\alpha\beta\sigma\tau,\varepsilon}}
\newcommand{\beps}{b^{\alpha\beta\sigma\tau,\varepsilon}}
\newcommand{\ceps}{c^{\alpha\beta\sigma\tau,\varepsilon}}
\renewcommand{\ae}{\ a.e. \ t\in(0,T)}
\newcommand{\aes}{\ a.e. \ \textrm{in} \ (0,T)}
\newcommand{\forallt}{\ \forall  \ t\in[0,T]}

\newcommand{\Upskl}{\Upsilon_{kl}}
\newcommand{\Upsij}{\Upsilon_{ij}}
\newcommand{\Upsab}{\Upsilon_{\alpha\beta}}
\newcommand{\Upsst}{\Upsilon_{\sigma\tau}}

\newcommand{\Upsdtres}{\Upsilon_{3 3}}
\newcommand{\dUpskl}{\dot{\Upsilon}_{kl}}
\newcommand{\dUpsij}{\dot{\Upsilon}_{ij}}
\newcommand{\dUpsab}{\dot{\Upsilon}_{\alpha\beta}}
\newcommand{\dUpsst}{\dot{\Upsilon}_{\sigma\tau}}

\newcommand{\dUpsdtres}{\dot{\Upsilon}_{3 3}}

\newcommand{\deb}{\rightharpoonup}
\newcommand{\en}{ \ \textrm{in} \ }
\newcommand{\on}{ \ \textrm{on} \ }
\newcommand{\into}{\int_{\omega}}
\newcommand{\intO}{\int_{\Omega}}

\newcommand{\ten}{(a^{\alpha \sigma}a^{\beta \tau} + a^{\alpha \tau}a^{\beta\sigma})}

\newcommand{\WLO}{H^1(0,T;L^2(\Omega))}
\newcommand{\WLOt}{H^1(0,T;[L^2(\Omega)]^3)}
\newcommand{\WLo}{H^1(0,T;L^2(\omega))}
\newcommand{\WHO}{H^1(0,T;H^1(\Omega))}
\newcommand{\WHOt}{H^1(0,T;[H^1(\Omega)]^3)}

\newcommand{\WHo}{H^1(0,T;H^1(\omega))}

\newcommand{\WVO}{H^1(0,T;V(\Omega))}
\newcommand{\LLO}{L^2(0,T;L^2(\Omega))}
\newcommand{\LLo}{L^2(0,T;L^2(\omega))}
\newcommand{\LHO}{L^2(0,T;H^1(\Omega))}
\newcommand{\LHo}{L^2(0,T;H^1(\omega))}
\newcommand{\WHMO}{H^1(0,T;H^{-1}(\Omega))}

\newcommand{\nWHOt}{_{H^1(0,T;[H^1(\Omega)]^3)}}

\begin{document}

\begin{frontmatter}

\title{Asymptotic Analysis of a  Viscoelastic Flexural Shell Model}


\author[mymainaddress]{G.~Casti\~neira\corref{mycorrespondingauthor}}
\ead{gonzalo.castineira@usc.es}
\author[mysecondaryaddress]{\'A. Rodr\'{\i}guez-Ar\'os}
\cortext[mycorrespondingauthor]{Corresponding author}
\ead{angel.aros@udc.es}

\address[mymainaddress]{Facultade de Matem\'aticas, Lope G\'omez de Marzoa, s/n.
              Campus sur, 15782,  Departamento de Matem\'atica Aplicada,
              Universidade de Santiago de Compostela, Spain }
\address[mysecondaryaddress]{ E.T.S. N\'autica e M\'aquinas
              Paseo de Ronda, 51, 15011, Departamento de Matem\'aticas, Universidade da Coru\~na, Spain}

\begin{abstract}
We consider a family of linearly viscoelastic shells with thickness $2\varepsilon$, clamped along a portion of their lateral face, all having the same middle surface $S=\btheta(\bar{\omega})\subset\mathbb{R}^3$, where $\omega\subset\mathbb{R}^2$ is a bounded and connected open set with a Lipschitz-continuous boundary $\gamma$. 
We show that, if the applied body force density is $O(\varepsilon^2)$ with respect to $\varepsilon$ and surface tractions density is $O(\var^3)$, the solution of the scaled variational problem in curvilinear coordinates, $\bu(\varepsilon)$, defined over the fixed domain $\Omega=\omega\times(-1,1)$, converges to a limit $\bu$ in $\WHOt$ as $\var\to 0$. Moreover, we prove that this limit is independent of the transverse variable. Furthermore, the average $\bar{\bu}= \frac1{2}\int_{-1}^{1}\bu dx_3$, which belongs to the space $H^{1}(0,T; V_F(\omega))$, where
\begin{align*}
V_F(\omega)&:= \{ \beeta=(\eta_i) \in H^1(\omega)\times H^1(\omega)\times H^2(\omega) ; 
\\& \qquad \eta_i=\d_\nu \eta_3=0 \ \textrm{on} \ \gamma_0, \gab(\beeta)=0 \en \omega \},
\end{align*}
satisfies   what we have identified as (scaled) two-dimensional equations of a viscoelastic flexural shell, which includes  a long-term memory that takes into account previous deformations. We finally provide convergence results which justify those equations.
\end{abstract}

\begin{keyword} 
 Asymptotic analysis \sep viscoelasticity \sep shells \sep flexural  \sep quasistatic problem.
\MSC[2010]  34K25 \sep  35Q74  \sep 34E05  \sep 34E10 \sep 41A60  \sep 74K25    \sep   74D05  \sep  35J15
\end{keyword}

\end{frontmatter}

\section{Introduction}\setcounter{equation}{0}

 In the last decades, many authors have applied the asymptotic methods in   three-dimensional elasticity problems in order to derive new reduced one-dimensional or two-dimensional models and justify the existing ones.  A complete theory regarding  elastic shells can be found in \cite{Ciarlet4b}, where models  for elliptic membranes, generalized membranes and flexural shells are presented. It contains a full description of the asymptotic procedure that leads to the corresponding sets of two-dimensional equations. Particularly,  the existence and uniqueness of solution of elastic elliptic membrane shell equations, can be found in \cite{CiarletLods} and in \cite{CiarletLods2}.  There, the two-dimensional elastic models are completely justified with convergence theorems. Besides, we can find the corresponding results for the elastic flexural shell problems in  \cite{CiarletLods4}.
 More recently in \cite{ArosObs,ArosObs2} the obstacle problem for an elastic elliptic membrane has been identified and justified as the limit problem for a family of unilateral contact problems of elastic elliptic shells by using asymptotic analysis. 

 However, a large number of actual physical and engineering problems have made it necessary the study of  models which  take into account effects such as hardening and memory of the material. An example of these are the  viscoelastic models (see for example \cite{DL,LC1990}). In some of these models, we can find   terms which  take into account the history of previous deformations or stresses of the body, known as  long-term memory. For a family of  shells made of a long-term memory viscoelastic material  we can find in  \cite{Viscoshells,Viscoshellsf,ViscoshellsK} the use of asymptotic analysis to justify with convergence results the limit two-dimensional membrane, flexural and Koiter equations.

  In this direction, to our knowledge, in \cite{intro2} we gave the first steps towards the justification of existing models of viscoelastic shells and  finding  new ones with the starting point being three-dimensional Kelvin-Voigt viscoelastic shell problems. By using the asymptotic expansion method, we found a rich variety of cases for the limit two-dimensional problems, depending on the geometry of the middle surface, the boundary conditions and the order of the applied forces. The most remarkable feature found was that from the asymptotic analysis of the three-dimensional problems  a long-term memory arose in the two-dimensional limit problems, represented by an integral with respect to the time variable. The aim of this paper is to mathematically justify these equations that we  identified in \cite{intro2} as the viscoelastic flexural shell problem, by presenting rigorous convergence results.

In this work we  justify  the two-dimensional equations of a viscoelastic flexural shell  where the  the boundary condition of place is considered in a portion lateral face of the shell:  
\begin{problem}\label{problema_ab_eps}
 Find $\bxi^\var:[0,T] \times\omega \longrightarrow \mathbb{R}^3$ such that,
    \begin{align*}\nonumber
    & \bxi^\var(t,\cdot)\in V_F(\omega), \forallt,\\ \nonumber
   &\frac{\var^3}{3}\int_{\omega} \aeps\rst(\bxi^\var(t))\rab(\beeta)\sqrt{a}dy +\frac{\var^3}{3}\int_{\omega}\beps\rst(\dot{\bxi}^\var(t))\rab(\beeta)\sqrt{a}dy
   \\\nonumber
    & \qquad- \frac{\var^3}{3}\int_0^te^{-k(t-s)}\into \ceps \rst(\bxi^\var(s))\rab(\beeta)\sqrt{a}dy ds 
    \\ 
   &\quad=\int_{\omega}p^{i,\var}(t)\eta_i\sqrt{a}dy \ \forall \beeta=(\eta_i)\in V_F(\omega), \aes,
    \\
    &\bxi^\var(0,\cdot)=\bxi^\var_0(\cdot),
   \end{align*}
   where,
   \begin{align*}
   & p^{i,\var}(t):=\int_{-\var}^{\var}\fb^{i,\var}(t)dx_3^\var +h_+^{i,\var}(t)+h_-^{i,\var}(t) \ \textrm{and} \ h_{\pm}^{i,\var}(t)=\bh^{i,\var}(t,\cdot,\pm \var),
   \end{align*}
  and where the contravariant components of the fourth order two-dimensional tensors $\aeps,$ $\beps, $ $\ceps$ are defined as  rescaled versions of two-dimensional fourth order tensors that we shall recall later in (\ref{tensor_a_bidimensional})--(\ref{tensor_c_bidimensional}). 
\end{problem}

 In what follows, we shall prove that the scaled three-dimensional unknown, $\bu(\var)$, converges as the small parameter $\var$ tends to zero to a limit, $\bu$,  independent of the transversal variable. Moreover, we find that this limit can be identified with  transversal average, $\bar{\bu}$, for all point of the middle surface of the shell. Furthermore, we prove that $\bar{\bu}$ is the unique solution of the Problem \ref{problema_ab_eps}, hence, the limit of the scaled unknown can be also identified with the solution of the two-dimensional problem, $\bxi$, defined over the middle surface of the shell. 
 
 We will follow the notation and style of \cite{Ciarlet4b}, where the linear elastic shells are studied.
For this reason, we shall  reference auxiliary results which apply in the same manner to the viscoelastic case. One of the major differences with respect to previous works in elasticity, consists on the time dependence, that will lead to ordinary differential equations that need to be solved in order to find the zeroth-order approach of the solution.

The structure of the paper is the following: in Section \ref{problema} 
we shall recall the three-dimensional viscoelastic  problem in Cartesian coordinates and  then, considering the problem for a family of viscoelastic shells of thickness $2\var$, we formulate the problem in curvilinear coordinates. In Section \ref{seccion_dominio_ind} we will use a projection map into a reference domain independent of the small parameter $\var$, we will introduce the scaled unknowns and forces and we present the assumptions on coefficients. In Section \ref{preliminares} we recall some technical results which will be needed in what follows.   In Section \ref{seccion_convergencia}, first we recall the results in \cite{intro2}, where, in particular, the two-dimensional equations for a viscoelastic flexural shell were studied. Then, we present the convergence results when the small parameter $\var$ tends to zero, which is the main result of this paper. After that, we present the convergence results in terms of de-scaled unknowns. In Section \ref{conclusiones} we shall present some conclusions, including a comparison between the viscoelastic models and the elastic case studied in \cite{Ciarlet4b} and comment about the convergence results regarding other cases.

\section{The three-dimensional linearly viscoelastic shell problem}\setcounter{equation}{0} \label{problema}

 We denote $\mathbb{S}^d$, where $d=2,3$ in practice, the space of second-order symmetric tensors on $\mathbb{R}^d$, while \textquotedblleft$\ \cdot$ \textquotedblright will represent the inner product and $|\cdot|$  the usual norm in $\mathbb{S}^d$ and  $\mathbb{R}^d$. In  what follows, unless the contrary is explicitly written, we will use summation convention on repeated indices. Moreover, Latin indices $i,j,k,l,...$, take their values in the set $\{1,2,3\}$, whereas Greek indices $\alpha,\beta,\sigma,\tau,...$, do it in the set  $\{1,2\}$. Also, we use standard notation for the Lebesgue and Sobolev spaces. Also, for a time dependent function $u$, we denote $\dot{u}$ the first derivative of $u$ with respect to the time variable. Recall that  $"\rightarrow"$ denotes strong convergence, while  $"\rightharpoonup" $ denotes weak convergence.
 
 
 Let ${\Omega}^*$ be a domain of $\mathbb{R}^3$, with a Lipschitz-continuous boundary ${\Gamma^*}=\d{\Omega^*}$. Let ${\bx^*}=({x}_i^*)$ be a generic point of  its closure $\bar{\Omega}^*$ and let ${\d}^*_i$ denote the partial derivative with respect to ${x}_i^*$. Let $d\bx^*$ denote the volume element in $\Omega^*$,  $d\Gamma^*$ denote the area element along $\Gamma^*$ and  $\bn^*$ denote the unit outer normal vector along $\Gamma^*$. Finally, let $\Gamma^*_0$  and $\Gamma_1^*$ be subsets of $\Gamma^*$ such that $meas(\Gamma_0^*)>0$ and $\Gamma^*_0 \cap \Gamma_1^*=\emptyset.$ 
 
  The set $\Omega^*$ is the region occupied by a deformable body in the absence of applied forces. We assume that this body is made of a Kelvin-Voigt viscoelastic material, which is homogeneous and isotropic,  so that the material is characterized by its Lam\'e coefficients   $\lambda\geq0, \mu>0$ and its viscosity coefficients, $\theta\geq 0,\rho\geq 0$ (see for instance \cite{DL,LC1990,Shillor}).

 Let $T>0$ be the time period of observation. Under the effect of applied forces, the  body is deformed and we denote by $u_i^*:[0,T]\times \bar{\Omega}^*\rightarrow \mathbb{R}^3$ the Cartesian components of the displacements field, defined as $\bu^*:=u_i^* \be^{i}:[0,T]\times\bar{\Omega}^* \rightarrow \mathbb{R}^3$, where $\{\be^i\}$ denotes the Euclidean canonical basis in $\mathbb{R}^3$. 
 Moreover, we consider that the displacement field vanishes on the set $\Gamma^*_0$. Hence, the  displacements field $\bu^*=(u_i^*):[0,T]\times\Omega^*\longrightarrow \mathbb{R}^3$ is solution of the following three-dimensional problem in Cartesian coordinates.
 
 \begin{problem}\label{problema_mecanico}
 Find $\bu^*=(u_i^*):[0,T]\times\Omega^*\longrightarrow \mathbb{R}^3$ such that,
 \begin{align}\label{equilibrio}
 -\d_j^*\sigma^{ij,*}(\bu^*)&=f^{i,*} \en \Omega^*, \\\label{Dirichlet}
 u_i^*&=0 \on \Gamma^*_0, \\\label{Neumann}
 \sigma^{ij,*}(\bu^*)n_j^*&=h^{i,*} \on \Gamma_1^*,\\ \label{condicion_inicial} 
 \bu^*(0,\cdot)&=\bu_0^* \en \Omega^*,
 \end{align}
 where the functions
 \begin{align*}
 \sigma^{ij,*}(\bu^*):=A^{ijkl,*}e_{kl}^*(\bu^*)+ B^{ijkl,*}e_{kl}^*(\dot{\bu}^*),
 \end{align*}
 are the  components of the linearized stress tensor field and where the functions
  \begin{align*} 
  & A^{ijkl,*}:= \lambda \delta^{ij}\delta^{kl} + \mu\left(\delta^{ik}\delta^{jl} + \delta^{il}\delta^{jk}\right) , 
  \\ 
  & B^{ijkl,*}:= \theta \delta^{ij}\delta^{kl} + \frac{\rho}{2}\left(\delta^{ik}\delta^{jl} + \delta^{il}\delta^{jk}\right) , 
 \end{align*}
  are the  components of the three-dimensional elasticity and viscosity fourth order tensors, respectively, and 
 \begin{align*}
  e^*_{ij}(\bu^*):= \frac1{2}(\d^*_ju^*_{i}+ \d^*_iu^*_{j}),
 \end{align*}
 designates the  components of the linearized strain tensor associated with the displacement field $\bu^*$of the set $\bar{\Omega}^*$.
 \end{problem}
 We now proceed to describe the equations in Problem \ref{problema_mecanico}. Expression (\ref{equilibrio}) is the equilibrium equation, where $f^{i,*}$ are the  components of the volumic force densities. The equality (\ref{Dirichlet}) is the Dirichlet condition of place, (\ref{Neumann}) is the Neumann condition, where $h^{i,*}$ are the  components of surface force densities and (\ref{condicion_inicial}) is the initial condition, where $\bu_0^*$ denotes the initial  displacements.
 
  Note that, for the sake of briefness, we omit the explicit dependence on the space and time variables when there is no ambiguity. Let us define the space of admissible unknowns,
 \begin{align*} 
 V(\Omega^*)=\{\bv^*=(v_i^*)\in [H^1(\Omega^*)]^3; \bv^*=\mathbf{\bcero} \ on \ \Gamma_0^*  \}.
 \end{align*}
 Therefore, assuming enough regularity,  the unknown  $\bu^*=(u_i^*)$ satisfies the following variational problem in Cartesian coordinates:
\begin{problem}\label{problema_cartesian}
Find $\bu^*=(u_i^*):[0,T]\times {\Omega}^* \rightarrow \mathbb{R}^3$  such that, 
\begin{align*} 
  \displaystyle  \nonumber
  & \bu^*(t,\cdot)\in V(\Omega^*) \forallt,
  \\ \nonumber 
   &\int_{\Omega^*}A^{ijkl,*}e^*_{kl}(\bu^*)e^*_{ij}(\bv^*) dx^*+ \int_{\Omega^*} B^{ijkl,*}e^*_{kl}(\dot{\bu}^*)e_{ij}^*(\bv^*)   dx^*
  \\ 
 & \quad= \int_{\Omega^*} f^{i,*} v_i^*  dx^* + \int_{\Gamma_1^*} h^{i,*} v_i^*  d\Gamma^* \quad \forall \bv^*\in V(\Omega^*), \aes,
  \\\displaystyle 
  & \bu^*(0,\cdot)= \bu_0^*(\cdot).
\end{align*}
\end{problem} 

Let us consider that $\Omega^*$ is a viscoelastic shell of thickness $2\var$. Now, we shall express the equations of the Problem \ref{problema_cartesian} in terms of  curvilinear coordinates. Let $\omega$ be a domain of $\mathbb{R}^2$, with a Lipschitz-continuous boundary $\gamma=\d\omega$. Let $\by=(y_\alpha)$ be a generic point of  its closure $\bar{\omega}$ and let $\d_\alpha$ denote the partial derivative with respect to $y_\alpha$. 

Let $\btheta\in\mathcal{C}^2(\bar{\omega};\mathbb{R}^3)$ be an injective mapping such that the two vectors $\ba_\alpha(\by):= \d_\alpha \btheta(\by)$ are linearly independent. These vectors form the covariant basis of the tangent plane to the surface $S:=\btheta(\bar{\omega})$ at the point $\btheta(\by)=\by^*.$  We can consider the two vectors $\ba^\alpha(\by)$ of the same tangent plane defined by the relations $\ba^\alpha(\by)\cdot \ba_\beta(\by)=\delta_\beta^\alpha$, that constitute the contravariant basis. We define the unit vector, 
\begin{align}\label{a_3}
\ba_3(\by)=\ba^3(\by):=\frac{\ba_1(\by)\wedge \ba_2(\by)}{| \ba_1(\by)\wedge \ba_2(\by)|},
\end{align} 
 normal vector to $S$ at the point $\btheta(\by)=\by^*$, where $\wedge$ denotes vector product in $\mathbb{R}^3.$ 

We can define the first fundamental form, given as metric tensor, in covariant or contravariant components, respectively, by
\begin{align*}
a_{\alpha\beta}:=\ba_\alpha\cdot \ba_\beta, \qquad a^{\alpha\beta}:=\ba^\alpha\cdot \ba^\beta,
\end{align*}
  the second fundamental form, given as curvature tensor, in covariant or mixed components, respectively, by
\begin{align*}
b_{\alpha\beta}:=\ba^3 \cdot \d_\beta \ba_\alpha, \qquad b_{\alpha}^\beta:=a^{\beta\sigma} b_{\sigma\alpha},
\end{align*}
and the Christoffel symbols of the surface $S$ by
\begin{align*}
\Gamma^\sigma_{\alpha\beta}:=\ba^\sigma\cdot \d_\beta \ba_\alpha.
\end{align*}

The area element along $S$ is $\sqrt{a}dy=dy^*$ where 
\begin{align}\label{definicion_a}
a:=\det (a_{\alpha\beta}).
\end{align}

Let $\gamma_0$ be a subset  of  $\gamma$, such that $meas (\gamma_0)>0$. 
For each $\varepsilon>0$, we define the three-dimensional domain $\Omega^\varepsilon:=\omega \times (-\varepsilon, \varepsilon)$ and  its boundary $\Gamae=\d\Omega^\var$. We also define  the following parts of the boundary, 
\begin{align*}
\Gamma^\varepsilon_+:=\omega\times \{\varepsilon\}, \quad \Gamma^\varepsilon_-:= \omega\times \{-\varepsilon\},\quad \Gamma_0^\varepsilon:=\gamma_0\times[-\varepsilon,\varepsilon].
\end{align*}

Let $\bx^\varepsilon=(x_i^\varepsilon)$ be a generic point of $\bar{\Omega}^\varepsilon$ and let $\d_i^\var$ denote the partial derivative with respect to $x_i^\varepsilon$. Note that $x_\alpha^\varepsilon=y_\alpha$ and $\d_\alpha^\varepsilon =\d_\alpha$. Let $\bTheta:\bar{\Omega}^\varepsilon\rightarrow \mathbb{R}^3$ be the mapping defined by
\begin{align} \label{bTheta}
\bTheta(\bx^\varepsilon):=\btheta(\by) + x_3^\varepsilon \ba_3(\by) \ \forall \bx^\varepsilon=(\by,x_3^\varepsilon)=(y_1,y_2,x_3^\varepsilon)\in\bar{\Omega}^\varepsilon.
\end{align}

The next theorem shows that if the injective mapping $\btheta:\bar{\omega}\rightarrow\mathbb{R}^3$ is smooth enough, the mapping $\bTheta:\bar{\Omega}^\var\rightarrow\mathbb{R}^3$ is also injective for $\var>0$ small enough (see Theorem 3.1-1, \cite{Ciarlet4b}).

\begin{theorem}\label{var_0}
Let $\omega$ be a domain in $\mathbb{R}^2$. Let $\btheta\in\mathcal{C}^2(\bar{\omega};\mathbb{R}^3)$ be an injective mapping such that the two vectors $\ba_\alpha=\d_\alpha\btheta$ are linearly independent at all points of $\bar{\omega}$ and let $\ba_3$,  defined  in (\ref{a_3}). Then, there exists $\var_0>0$ such that $\forall \var_1$, $0<\var_1\leq \var_0$  and  the mapping $\bTheta:\bar{\Omega}_1 \rightarrow\mathbb{R}^3$ defined by
\begin{align*}
\bTheta(\by,x_3):=\btheta(\by) + x_3 \ba_3(\by) \ \forall (\by,x_3)\in\bar{\Omega}_1, \ \textrm{where} \ \Omega_1:=\omega\times(-\var_1,\var_1),
\end{align*}
is a $\mathcal{C}^1-$ diffeomorphism from $\bar{\Omega}_1$ onto $\bTheta(\bar{\Omega}_1)$ and $\det (\bg_1,\bg_2,\bg_3)>0$ in $\bar{\Omega}_1$, where $\bg_i:=\d_i\bTheta$. 
\end{theorem}

As a consequence, for each $\var$, $0<\var\le\var_0$, the set $\bTheta(\bar{\Omega}^\var)=\bar{\Omega}^*$ is the reference configuration of a viscoelastic shell, with middle surface $S=\btheta(\bar{\omega})$ and thickness $2\varepsilon>0$.
Furthermore for $\varepsilon>0,$ $\bg_i^\varepsilon(\bx^\varepsilon):=\d_i^\varepsilon\bTheta(\bx^\varepsilon)$ are linearly independent and the mapping $\bTheta:\bar{\Omega}^\varepsilon\rightarrow \mathbb{R}^3$ is injective for all $\var$, $0<\var\le\var_0$, as a consequence of injectivity of the mapping $\btheta$. Hence, the three vectors $\bg_i^\varepsilon(\bx^\varepsilon)$ form the covariant basis of the tangent space at the point $\bx^*=\bTheta(\bx^\varepsilon)$ and $\bg^{i,\varepsilon}(\bx^\varepsilon) $ defined by the relations $\bg^{i,\varepsilon}\cdot \bg_j^\varepsilon=\delta_j^i$ form the contravariant basis at the point $\bx^*=\bTheta(\bx^\varepsilon)$. We define the metric tensor, in covariant or contravariant components, respectively, by
\begin{align*}
 g_{ij}^\varepsilon:=\bg_i^\varepsilon \cdot \bg_j^\varepsilon,\quad g^{ij,\varepsilon}:=\bg^{i,\varepsilon} \cdot \bg^{j,\varepsilon},
\end{align*}
and Christoffel symbols by
\begin{align} \label{simbolos3D}
\Gamma^{p,\varepsilon}_{ij}:=\bg^{p,\varepsilon}\cdot\d_i^\varepsilon \bg_j^\varepsilon. 
\end{align}

The volume element in the set $\bTheta(\bar{\Omega}^\varepsilon)=\bar{\Omega}^*$ is $\sqrt{g^\varepsilon}dx^\var=dx^*$ and the surface element in $\bTheta(\Gamma^\varepsilon)=\Gamma^*$ is $\sqrt{g^\varepsilon}d\Gamae =d\Gamma^*$  where
\begin{align} \label{g}
g^\varepsilon:=\det (g^\varepsilon_{ij}).
\end{align} 
Therefore, for a field ${\bv}^*$ defined in $\bTheta(\bar{\Omega}^\var)=\bar{\Omega}^*$, we define its covariant curvilinear coordinates $v_i^\var$ by
\begin{equation*}
{\bv}^*({\bx}^*)={v}^*_i({\bx}^*){\be}^i=:v_i^\var(\bx^\var)\bg^i(\bx^\var),\ {\rm with}\ {\bx}^*=\bTheta(\bx^\var).
\end{equation*}


Besides, we denote by $u_i^\varepsilon:[0,T]\times \bar{\Omega}^\varepsilon \rightarrow \mathbb{R}^3$ the covariant components of the displacements field, that is  $\bUcal^\var:=u_i^\varepsilon \bg^{i,\varepsilon}:[0,T]\times\bar{\Omega}^\varepsilon \rightarrow \mathbb{R}^3$ . For simplicity, we define the vector field $\bu^\varepsilon=(u_i^\varepsilon):[0,T]\times {\Omega}^\varepsilon \rightarrow \mathbb{R}^3$ which will be denoted vector of unknowns.

Recall that we assumed that the shell is subjected to a boundary condition of place; in particular that the displacements field vanishes in  $\bTheta(\Gamma_0^\varepsilon)=\Gamma_0^*$,  this is,  on a portion of the lateral face of the shell.

Accordingly, let us define the space of admissible unknowns,
\begin{align*} 
V(\Omega^\varepsilon)=\{\bv^\varepsilon=(v_i^\varepsilon)\in [H^1(\Omega^\varepsilon)]^3; \bv^\varepsilon=\mathbf{\bcero} \ on \ \Gamma_0^\varepsilon  \}.
\end{align*}

This is a real Hilbert space with the induced inner product of $[H^1(\Omega^\var)]^3$. The corresponding norm  is denoted by $||\cdot||_{1,\Omega^\var}$. 

Therefore, we can find the expression of the Problem \ref{problema_cartesian} in curvilinear coordinates (see \cite{Ciarlet4b} for details). Hence, the `` displacements " field $\bu^\var=(u_i^\var)$ verifies the following variational problem of a three-dimensional viscoelastic shell in curvilinear coordinates:

\begin{problem}\label{problema_eps}
Find $\bu^\varepsilon=(u_i^\varepsilon):[0,T]\times {\Omega}^\varepsilon \rightarrow \mathbb{R}^3$  such that, 
\begin{align*} 
  \displaystyle  \nonumber
  & \bu^\varepsilon(t,\cdot)\in V(\Omega^\varepsilon) \forallt,
  \\ \nonumber 
   &\int_{\Omega^\varepsilon}A^{ijkl,\varepsilon}e^\varepsilon_{k||l}(\bu^\varepsilon)e^\varepsilon_{i||j}(\bv^\varepsilon)\sqrt{g^\varepsilon} dx^\varepsilon+ \int_{\Omega^\varepsilon} B^{ijkl,\varepsilon}e^\varepsilon_{k||l}(\dot{\bu}^\varepsilon)e_{i||j}^\var(\bv^\varepsilon) \sqrt{g^\varepsilon}  dx^\varepsilon
  \\ 
 & \quad= \int_{\Omega^\varepsilon} f^{i,\varepsilon} v_i^\varepsilon \sqrt{g^\varepsilon} dx^\varepsilon + \int_{\Gamma_+^\varepsilon\cup\Gamma_-^\varepsilon} h^{i,\varepsilon} v_i^\varepsilon\sqrt{g^\varepsilon}  d\Gamma^\varepsilon  \quad \forall \bv^\varepsilon\in V(\Omega^\varepsilon), \aes,
  \\\displaystyle 
  & \bu^\varepsilon(0,\cdot)= \bu_0^\varepsilon(\cdot),
\end{align*}
\end{problem}
 where the functions
  \begin{align}\label{TensorAeps}
  & A^{ijkl,\varepsilon}:= \lambda g^{ij,\varepsilon}g^{kl,\varepsilon} + \mu(g^{ik,\varepsilon}g^{jl,\varepsilon} + g^{il,\varepsilon}g^{jk,\varepsilon} ), 
  \\ \label{TensorBeps}
  & B^{ijkl,\varepsilon}:= \theta g^{ij,\varepsilon}g^{kl,\varepsilon} + \frac{\rho}{2}(g^{ik,\varepsilon}g^{jl,\varepsilon} + g^{il,\varepsilon}g^{jk,\varepsilon} ), 
 \end{align}
  are the contravariant components of the three-dimensional elasticity and viscosity tensors, respectively. We assume that the Lam\'e coefficients   $\lambda\geq0, \mu>0$ and the viscosity coefficients $\theta\geq 0,\rho\geq 0$  are all independent of $\var$. Moreover, the terms
 \begin{align*}
  e^\varepsilon_{i||j}(\bu^\var):= \frac1{2}(u^\varepsilon_{i||j}+ u^\varepsilon_{j||i})=\frac1{2}(\d^\varepsilon_ju^\varepsilon_i + \d^\varepsilon_iu^\varepsilon_j) - \Gamma^{p,\varepsilon}_{ij}u^\varepsilon_p,
 \end{align*}
 designate the covariant components of the linearized strain tensor associated with the displacement field $\bUcal^\var$of the set $\bTheta(\bar{\Omega}^\varepsilon)$.  Moreover, $f^{i,\var}$ denotes the contravariant components of the volumic force densities, $h^{i,\var}$ denotes contravariant components of surface force densities and $\bu_0^\var$ denotes the initial ``displacements" (actually, the initial displacement is $\bUcal^\var(0)=:\bUcal_0^\var=(u_0^\var)_i\bg^{i,\var}$).

 Note that the following additional relations are satisfied,
 \begin{align}\nonumber
 \Gamma^{3,\varepsilon}_{\alpha 3}=\Gamma^{p,\varepsilon}_{33}&=0  \ \textrm{in} \ \bar{\Omega}^\varepsilon, \\
\label{tensor_terminos_nulos}
 A^{\alpha\beta\sigma 3,\varepsilon}=A^{\alpha 333,\varepsilon}=B^{\alpha\beta\sigma 3 , \varepsilon}&=B^{\alpha 333, \varepsilon}=0 \ \textrm{in} \ \bar{\Omega}^\varepsilon,
 \end{align}
 as a consequence of the definition of $\bTheta$ in (\ref{bTheta}). 
 
The existence and uniqueness of solution of the Problem \ref{problema_eps} for $\var>0$ small enough, established in the following theorem,  was proved in \cite{intro2}: 
 \begin{theorem}\label{Thexistunic}
  Let $\Omega^\var$ be a domain in $\mathbb{R}^3$ defined previously in this section and let $\bTheta$ be a  $\mathcal{C}^2$-diffeomorphism of $\bar{\Omega}^\var$ in its image $\bTheta(\bar{\Omega}^\var)$, such that the three vectors $\bg_i^\var(\bx)=\d_i^\var\bTheta(\bx^\var)$ are linearly independent for all $\bx^\var\in\bar{\Omega}^\var$. Let $\Gamma_0^\var$ be a $d\Gamma^\var$-measurable subset of $\Gamma^\var=\d\Omega^\var$ such that  $meas(\Gamma_0^\var)>0.$
  Let $\fb^{i,\var}\in L^{2}(0,T; L^2(\Omega^\var)) $, $\bh^{i,\var}\in L^{2}(0,T; L^2(\Gamma_1^\var))$, where $\Gamma_1^\var:= \Gamma_+^\var\cup\Gamma_-^\var$. Let  $\bu_0^\var\in V(\Omega^\var). $ Then, there exists a unique solution $\bu^\var=(u_i^\var):[0,T]\times\Omega^\var \rightarrow \mathbb{R}^3$ satisfying the Problem \ref{problema_eps}. Moreover, $\bu^\var\in H^{1}(0,T;V(\Omega^\var))$. In addition to that, if $\dot{\fb}^{i,\var}\in L^{2}(0,T; L^2(\Omega^\var)) $, $\dot{\bh}^{i,\var}\in L^{2}(0,T; L^2(\Gamma_1^\var))$, then $\bu^\var\in H^{2}(0,T;V(\Omega^\var))$.
 \end{theorem}

\section{The scaled three-dimensional shell problem}\setcounter{equation}{0} \label{seccion_dominio_ind}

For convenience, we consider a reference domain independent of the small parameter $\var$. Hence, let us define the three-dimensional domain $\Omega:=\omega \times (-1, 1) $ and  its boundary $\Gamma=\d\Omega$. We also define the following parts of the boundary,
 \begin{align*}
 \Gamma_+:=\omega\times \{1\}, \quad \Gamma_-:= \omega\times \{-1\},\quad \Gamma_0:=\gamma_0\times[-1,1].
 \end{align*}
 Let $\bx=(x_1,x_2,x_3)$ be a generic point in $\bar{\Omega}$ and we consider the notation $\d_i$ for the partial derivative with respect to $x_i$. We define the following projection map, 
 \begin{align*}
 \pi^\varepsilon:\bx=(x_1,x_2,x_3)\in \bar{\Omega} \longrightarrow \pi^\varepsilon(\bx)=\bx^\varepsilon=(x_i^\varepsilon)=(x_1^\var,x_2^\var,x_3^\var)=(x_1,x_2,\varepsilon x_3)\in \bar{\Omega}^\varepsilon,
 \end{align*}
 hence, $\d_\alpha^\varepsilon=\d_\alpha $  and $\d_3^\varepsilon=\frac1{\varepsilon}\d_3$. We consider the scaled unknown $\bu(\varepsilon)=(u_i(\varepsilon)):[0,T]\times \bar{\Omega}\longrightarrow \mathbb{R}^3$ and the scaled vector fields $\bv=(v_i):\bar{\Omega}\longrightarrow \mathbb{R}^3 $ defined as
 \begin{align*}
 u_i^\varepsilon(t,\bx^\varepsilon)=:u_i(\varepsilon)(t,\bx) \ \textrm{and} \ v_i^\varepsilon(\bx^\varepsilon)=:v_i(\bx) \ \forall \bx^\varepsilon=\pi^\varepsilon(\bx)\in \bar{\Omega}^\varepsilon, \ \forall \ t\in[0,T].
 \end{align*}

 Also, let the functions, $\Gamma_{ij}^{p,\varepsilon}, g^\varepsilon, A^{ijkl,\varepsilon}, B^{ijkl,\varepsilon}$ defined in (\ref{simbolos3D}), (\ref{g}), (\ref{TensorAeps}) and (\ref{TensorBeps}), be associated with the functions $\Gamma_{ij}^p(\varepsilon), g(\varepsilon), A^{ijkl}(\varepsilon), B^{ijkl}(\varepsilon)$ defined by
  \begin{align} \label{escalado_simbolos}
  &\Gamma_{ij}^p(\varepsilon)(\bx):=\Gamma_{ij}^{p,\varepsilon}(\bx^\varepsilon),
  \\\label{escalado_g}
  & g(\varepsilon)(\bx):=g^\varepsilon(\bx^\varepsilon),
  \\\label{tensorA_escalado}
  & A^{ijkl}(\varepsilon)(\bx):=A^{ijkl,\varepsilon}(\bx^\varepsilon),
  \\\label{tensorB_escalado}
  & B^{ijkl}(\varepsilon)(\bx):=B^{ijkl,\varepsilon}(\bx^\varepsilon),
  \end{align}
 for all $\bx^\varepsilon=\pi^\varepsilon(\bx)\in\bar{\Omega}^\varepsilon$. For all $\bv=(v_i)\in [H^1(\Omega)]^3$, let there be associated the scaled linearized strains components $\eij(\var)(\bv)\in L^2(\Omega)$, defined by
\begin{align} \label{eab}
&\eab(\varepsilon;\bv):=\frac{1}{2}(\d_\beta v_\alpha + \d_\alpha v_\beta) - \Gamma_{\alpha\beta}^p(\varepsilon)v_p,\\ 
  & \eatres(\varepsilon;\bv):=\frac{1}{2}\left(\frac{1}{\var}\d_3 v_\alpha + \d_\alpha v_3\right) - \Gamma_{\alpha 3}^p(\varepsilon)v_p,\\ \label{edtres}
  & \edtres(\varepsilon;\bv):=\frac1{\varepsilon}\d_3v_3.
\end{align}
Note that with these definitions it is verified that
\begin{align*}
\eij^\var(\bv^\var)(\pi^\var(\bx))=\eij(\var;\bv)(\bx) \ \forall\bx\in\Omega.
\end{align*}

 \begin{remark} The functions $\Gamma_{ij}^p(\varepsilon), g(\varepsilon), A^{ijkl}(\varepsilon), B^{ijkl}(\varepsilon)$ converge in $\mathcal{C}^0(\bar{\Omega})$ when $\varepsilon$ tends to zero. However,  $\eatres$ and $\edtres$ are not well defined, hence, this case  leads to a singular problem. 
 \end{remark}
 
 \begin{remark}When we consider
 $\varepsilon=0$ the functions will be defined with respect to $\by\in\bar{\omega}$. We shall distinguish the three-dimensional Christoffel symbols from the two-dimensional ones by using  $\Gamma_{\alpha \beta}^\sigma(\varepsilon)$ and $ \Gamma_{\alpha\beta}^\sigma$, respectively.
 \end{remark}

The next result is an adaptation of $(b)$ in Theorem 3.3-2, \cite{Ciarlet4b} to the viscoelastic case.  We will study the asymptotic behaviour of the scaled contravariant components $A^{ijkl}(\var), B^{ijkl}(\var)$ of the three-dimensional elasticity and viscosity tensors defined in (\ref{tensorA_escalado})--(\ref{tensorB_escalado}), as $\var\rightarrow0$.  We show their uniform positive definiteness  not only with respect to $\bx\in\bar{\Omega}$, but also with respect to $\var$, $0<\var\leq\var_0$. Finally, their limits are functions of $\by\in\bar{\omega}$ only, that is, independent of the transversal variable $x_3$.
\begin{theorem} \label{Th_comportamiento asintotico}
Let $\omega$  be a domain in $\mathbb{R}^2$ and let $\btheta\in\mathcal{C}^2(\bar{\omega};\mathbb{R}^3)$ be an injective mapping such that the two vectors $\ba_\alpha=\d_\alpha\btheta$ are linearly independent at all points of $\bar{\omega}$, let $a^{\alpha\beta}$ denote the contravariant components of the metric tensor of $S=\btheta(\bar{\omega})$. In addition to that, let the other assumptions on the mapping $\btheta$ and the definition of $\var_0$ be as in Theorem \ref{var_0}. The contravariant components $A^{ijkl}(\var), B^{ijkl}(\var)$ of the scaled three-dimensional elasticity and viscosity tensors, respectively, defined in (\ref{tensorA_escalado})--(\ref{tensorB_escalado}) satisfy
\begin{align*}
A^{ijkl}(\var)= A^{ijkl}(0) + O(\var) \ \textrm{and} \ A^{\alpha\beta\sigma 3}(\var)=A^{\alpha 3 3 3}(\var)=0, \\
B^{ijkl}(\var)= B^{ijkl}(0) + O(\var) \ \textrm{and} \ B^{\alpha\beta\sigma 3}(\var)=B^{\alpha 3 3 3}(\var)=0 ,
\end{align*}
for all $\var$, $0<\var \leq \var_0$, 
 and
\begin{align*}
A^{\alpha\beta\sigma\tau}(0)&= \lambda a^{\alpha\beta}a^{\sigma\tau} + \mu(a^{\alpha\sigma}a^{\beta\tau} + a^{\alpha\tau}a^{\beta\sigma}), & A^{\alpha\beta 3 3}(0)&= \lambda a^{\alpha\beta},
\\
 A^{\alpha 3\sigma 3}(0)&=\mu a^{\alpha\sigma} ,& A^{33 3 3}(0)&= \lambda + 2\mu,
 \\
A^{\alpha\beta\sigma 3}(0) &=A^{\alpha 333}(0)=0,
\\
B^{\alpha\beta\sigma\tau}(0)&= \theta a^{\alpha\beta}a^{\sigma\tau} + \frac{\rho}{2}(a^{\alpha\sigma}a^{\beta\tau} + a^{\alpha\tau}a^{\beta\sigma}),& B^{\alpha\beta 3 3}(0)&= \theta a^{\alpha\beta},
\\
 B^{\alpha 3\sigma 3}(0)&=\frac{\rho}{2} a^{\alpha\sigma} ,& B^{33 3 3}(0)&= \theta + \rho, 
\\
B^{\alpha\beta\sigma 3}(0) &=B^{\alpha 333}(0)=0.
\end{align*}
 Moreover, there exist two constants $C_e>0$ and $C_v>0$, independent of the variables and $\var$, such that 
  \begin{align} \label{elipticidadA_eps}
  \sum_{i,j}|t_{ij}|^2\leq C_e A^{ijkl}(\varepsilon)(\bx)t_{kl}t_{ij},\\\label{elipticidadB_eps}
  \sum_{i,j}|t_{ij}|^2 \leq C_v B^{ijkl}(\varepsilon)(\bx)t_{kl}t_{ij},
  \end{align}
 for all $\var$, $0<\var\leq\var_0$, for all $\bx\in\bar{\Omega}$ and all $\bt=(t_{ij})\in\mathbb{S}^3$. 
\end{theorem}

 \begin{remark}
 Notice that, by the asymptotic behaviour of tensors $\left(A^{ijkl}(\varepsilon)\right)$  and  $\left(B^{ijkl}(\varepsilon)\right)$  from Theorem \ref{Th_comportamiento asintotico},  if we take the limit when $\var\to0$ in (\ref{elipticidadA_eps}) and (\ref{elipticidadB_eps}), we find, respectively, that 
  \begin{align} \label{elipticidadA_eps0}
   \sum_{i,j}|t_{ij}|^2\leq C_e A^{ijkl}(0)(\bx)t_{kl}t_{ij}, \quad 
   \sum_{i,j}|t_{ij}|^2 \leq C_v B^{ijkl}(0)(\bx)t_{kl}t_{ij},
   \end{align}
  for all $\bx\in\bar{\Omega}$ and all $\bt=(t_{ij})\in\mathbb{S}^3$.
 \end{remark}

 \begin{remark}
 Note that the proof for the scaled viscosity tensor $\left(B^{ijkl}(\varepsilon)\right)$ would follow the steps of the proof for the elasticity tensor $\left(A^{ijkl}(\var)\right)$ in Theorem 3.3-2, \cite{Ciarlet4b}, since from a quality point of view their expressions differ in replacing the Lam\'e constants by the two viscosity coefficients. 
 \end{remark}

 Let the scaled applied forces $\bbf^i(\varepsilon):[0,T]\times \Omega\longrightarrow \mathbb{R}^3$ and  $\bbh^i(\varepsilon):[0,T]\times (\Gamma_+\cup\Gamma_-)\longrightarrow \mathbb{R}^3$ be defined by
   \begin{align*}
  \bbf^\var&=(f^{i,\varepsilon})(t,\bx^\varepsilon)=:\bbf(\var)= (f^i(\varepsilon))(t,\bx) 
  \\ \nonumber
  &\forall \bx\in\Omega, \ \textrm{where} \ \bx^\varepsilon=\pi^\varepsilon(\bx)\in \Omega^\varepsilon \ \textrm{and} \ \forall t\in[0,T], \\ 
   \bbh^\var&=(h^{i,\varepsilon})(t,\bx^\varepsilon)=:\bbh(\var)= (h^i(\varepsilon))(t,\bx) 
   \\ \nonumber 
   &\forall \bx\in\Gamma_+\cup\Gamma_-, \ \textrm{where} \ \bx^\varepsilon=\pi^\varepsilon(\bx)\in \Gamma_+^\varepsilon\cup\Gamma_-^\varepsilon \ \textrm{and} \ \forall t\in[0,T].
   \end{align*}
   Also, we introduce $\bu_0(\var): \Omega \longrightarrow \mathbb{R}^3$ by
   \begin{align*}
   \bu_0(\var)(\bx):=\bu_0^\var(\bx^\var) \ \forall \bx\in\Omega, \ \textrm{where} \ \bx^\varepsilon=\pi^\varepsilon(\bx)\in \Omega^\varepsilon,
   \end{align*}
   and define the space
   \begin{align*} 
   V(\Omega):=\{\bv=(v_i)\in [H^1(\Omega)]^3; \bv=\mathbf{0} \ on \ \Gamma_0\},
   \end{align*}
  which is a Hilbert space, with associated norm denoted by $||\cdot||_{1,\Omega}$.

  We assume that the scaled applied forces are given by
  \begin{align*} 
  & \bbf(\varepsilon)(t, \bx)=\varepsilon^2\bbf^2(t,\bx) \ \forall \bx\in \Omega \ \textrm{and} \ \forall t\in[0,T], \\ 
  & \bbh(\varepsilon)(t, \bx)=\varepsilon^{3}\bbh^{3}(t,\bx) \ \forall  \bx\in \Gamma_+\cup\Gamma_- \ \textrm{and} \ \forall t\in[0,T],
  \end{align*}
 where $\bbf^2$ and $\bbh^{3}$ are functions independent of $\var$. Then, the scaled variational problem  can   be written as follows:
  \begin{problem}\label{problema_orden_fuerzas}
   Find $\bu(\varepsilon):[0,T]\times\Omega\longrightarrow \mathbb{R}^3$ such that,
   \begin{align} \nonumber
   & \bu(\varepsilon)(t,\cdot)\in V(\Omega) \forallt, \\ \nonumber
       &\int_{\Omega}A^{ijkl}(\varepsilon)e_{k||l}(\varepsilon,\bu(\varepsilon))e_{i||j}(\varepsilon,\bv)\sqrt{g(\varepsilon)} dx
    + \int_{\Omega} B^{ijkl}(\varepsilon)e_{k||l}(\varepsilon,\dot{\bu}(\varepsilon))e_{i||j}(\varepsilon,\bv) \sqrt{g(\varepsilon)}  dx
      \\  \label{ecuacion_orden_fuerzas}
      &\quad= \int_{\Omega} \var^2 \fb^{i,2}v_i \sqrt{g(\varepsilon)} dx + \int_{\Gamma_+\cup\Gamma_-} \var^{2}\bh^{i,3} v_i\sqrt{g(\varepsilon)}  d\Gamma  \quad \forall \bv\in V(\Omega), \aes,
      \\\displaystyle \nonumber
      & \bu(\var)(0,\cdot)= \bu_0(\var)(\cdot).
     \end{align} 
   \end{problem}
   From now on,   for each $\var>0$, we shall use  the shorter notation $\eij(\var)\equiv\eij(\varepsilon;\bu(\varepsilon))$ and $\deij(\var)\equiv\eij(\varepsilon;\dot{\bu}(\varepsilon))$, for its time derivative. 
   We recall the existence and uniqueness of the Problem \ref{problema_orden_fuerzas} in the following theorem whose proof can be found in  \cite{intro2}:

 \begin{theorem}\label{Theorema_exist_escalado_sin_orden} 
  Let $\Omega$ be a domain in $\mathbb{R}^3$ defined previously in this section  and let $\bTheta$ be a $\mathcal{C}^2$-diffeomorphism of $\bar{\Omega}$ onto its image $\bTheta(\bar{\Omega})$, such that the three vectors $\bg_i=\d_i\bTheta(\bx)$ are linearly independent for all $\bx\in\bar{\Omega}$. Let $\fb^{i}(\var)\in L^{2}(0,T; L^2(\Omega)) $, $\bh^{i}(\var)\in L^{2}(0,T; L^2(\Gamma_1))$, where $\Gamma_1:= \Gamma_+\cup\Gamma_-$. Let  $\bu_0(\var)\in V(\Omega). $ Then, there exists a unique solution $\bu(\var)=(u_i(\var)):[0,T]\times\Omega \rightarrow \mathbb{R}^3$ satisfying the Problem \ref{problema_orden_fuerzas}. Moreover $\bu(\var)\in H^{1}(0,T;V(\Omega))$. In addition to that, if $\dot{\fb}^{i}(\var)\in L^{2}(0,T; L^2(\Omega)) $, $\dot{\bh}^{i}(\var)\in L^{2}(0,T; L^2(\Gamma_1))$, then $\bu(\var)\in H^{2}(0,T;V(\Omega))$.
  \end{theorem}

\section{Technical preliminaries}\setcounter{equation}{0} \label{preliminares}

Concerning geometrical and mechanical preliminaries, we shall present some theorems, which will be used in the following sections. 
First, we recall the Theorem 3.3-1, \cite{Ciarlet4b}.

\begin{theorem} \label{Th_simbolos2D_3D}
Let $\omega$ be a domain in $\mathbb{R}^2$, let $\btheta\in\mathcal{C}^3(\bar{\omega};\mathcal{R}^3)$ be an injective mapping such that the two vectors $\ba_\alpha=\d_\alpha\btheta$ are linearly independent at all points of $\bar{\omega}$ and let $\var_0>0$ be as in Theorem \ref{var_0}. The functions $\Gamma^p_{ij}(\var)=\Gamma^p_{ji}(\var)$ and $g(\var)$ are defined in (\ref{escalado_simbolos})--(\ref{escalado_g}), the functions $b_{\alpha\beta}, b_\alpha^\sigma, \Gamma_{\alpha\beta}^\sigma,a$, are defined in Section \ref{problema} and the covariant derivatives $b_\beta^\sigma|_\alpha$ are defined by
\begin{align} \label{b_barra}
b_\beta^\sigma|_\alpha:=\d_\alpha b_\beta^\sigma +\Gamma^\sigma_{\alpha\tau}b_\beta^\tau - \Gamma^\tau_{\alpha\beta}b^\sigma_\tau.
\end{align}
The functions $b_{\alpha\beta}, b_\alpha^\sigma, \Gamma_{\alpha\beta}^\sigma, b_\beta^\sigma|_\alpha$ and $a$ are identified with functions in $\mathcal{C}^0(\bar{\Omega})$. Then
\begin{align*}
\begin{aligned}[c]
 \Gamma_{\alpha\beta}^\sigma(\var)&=  \Gamma_{\alpha\beta}^\sigma -\var x_3b_\beta^\sigma|_\alpha + O(\var^2), \\
  \d_3 \Gamma_{\alpha\beta}^p(\var)&= O(\var), 
   \\
   \Gamma_{\alpha3}^3(\var)&=\Gamma_{33}^p(\var)=0,
\end{aligned}
\qquad
\begin{aligned}[c]
 \Gamma_{\alpha\beta}^3(\var)&=b_{\alpha\beta} - \var x_3 b_\alpha^\sigma b_{\sigma\beta}, 
 \\
 \Gamma_{\alpha3}^\sigma(\var)& = -b_\alpha^\sigma - \var x_3 b_\alpha^\tau b_\tau^\sigma + O(\var^2), 
\\
 g(\varepsilon)&=a + O(\varepsilon),
\end{aligned}
\end{align*}
for all $\var$, $0<\var\leq\var_0$, where the order symbols $O(\var)$ and $O(\var^2)$  are meant with respect to the norm $||\cdot||_{0,\infty,\bar{\Omega}}$ defined by
\begin{align*} 
||w||_{0,\infty,\bar{\Omega}}=\sup \{|w(\bx)|; \bx\in\bar{\Omega}\}.
\end{align*}
 Finally, there exist constants $a_0, g_0$ and $g_1$ such that
 \begin{align} \nonumber
 & 0<a_0\leq a(\by) \ \forall \by\in \bar{\omega},
 \\ \label{g_acotado}
 & 0<g_0\leq g(\varepsilon)(\bx) \leq g_1 \ \forall \bx\in\bar{\Omega} \ \textrm{and} \ \forall \ \var, 0<\varepsilon\leq \varepsilon_0.
 \end{align}
\end{theorem}


We now include the following result that will be used repeatedly in what follows (see Theorem 3.4-1, \cite{Ciarlet4b}, for details).

\begin{theorem} \label{th_int_nula}
 Let $\omega$ be a domain in $\mathbb{R}^2$ with boundary $\gamma$, let $\Omega=\omega\times (-1,1)$, and let $g\in L^p(\Omega)$, $p>1$, be a function such that 
 \begin{align*}
 \intO g \d_3v dx=0, \ \textrm{for all} \ v\in \mathcal{C}^{\infty}(\bar{\Omega}) \ \textrm{with} \ v=0 \on \gamma\times[-1,1]. 
 \end{align*}
 Then $g=0.$
\end{theorem}
\begin{remark}
This result holds if $\intO g \d_3v dx=0$ for all $v\in H^1(\Omega)$ such that $v=0$ in $\Gamma_0$. It is in this way that we will use this result in the following.
\end{remark}


We now introduce the average with respect to the transversal variable, which plays a major role in this study. To that end, let $\bv$ represent real or vectorial functions defined  almost everywhere over $\Omega=\omega\times (-1,1)$. We define the transversal average as
    \begin{align*}
    \bar{\bv}(\by)=\frac1{2}\int_{-1}^{1}\bv(\by,x_3)dx_3
    \end{align*}
    for almost all $\by\in\omega$.
 Given $\beeta=(\eta_i)\in [H^1(\omega)]^3,$ let
   \begin{align} \label{def_gab}
   \gab(\beeta):= \frac{1}{2}(\d_\beta\eta_\alpha + \d_\alpha\eta_\beta) - \Gamma_{\alpha\beta}^\sigma\eta_\sigma -  b_{\alpha\beta}\eta_3,
   \end{align}
   denote the covariant components of the linearized change of metric tensor associated with a displacement field $\eta_i\ba^i$ of the surface $S$.
Next theorem will show some results related with the transversal averages that will be useful in the next section.  
    \begin{theorem}\label{Th_medias} 
    Let $\omega$ be a domain in $\mathbb{R}^2$, let $\Omega=\omega\times(-1,1)$ and $T>0$.
    \begin{enumerate}[label={{(\alph*)}}, leftmargin=0em ,itemindent=3em]
    \item Let $v\in\WLO$. Then $\bar{v}(\by)$ is finite for almost all  $\by\in\omega$,  belongs to $\WLo$, and
    \begin{align}\nonumber
    |\bar{v}|_{\WLo}\leq\frac{1}{\sqrt{2}}|v|_{\WLO}.
    \end{align}
    If $\d_3v=0$ in the distributions sense $\left(\int_{\Omega}v \d_3{\varphi} dx=0 \ \forall {\varphi} \in \mathcal{D}(\Omega) \right)$ then $v$ does not depend on $x_3$ and
    \begin{align}\nonumber
    v(\by,x_3)=\bar{v}(\by) \ \textrm{for almost all } \ (\by,x_3)\in\Omega.
    \end{align}
    \item Let $v\in \WHO$. Then $\bar{v}\in \WHo$, $\d_\alpha\bar{v}=\overline{\d_\alpha v}$ and
    \begin{align}\nonumber
    ||\bar{v}||_{\WHo}\leq\frac{1}{\sqrt{2}} ||v||_{\WHO}.
    \end{align}
    Let $\gamma_0$ be a subset $\d\gamma$-measurable of $\gamma$. If $v=0$ on $\gamma_0\times[-1,1]$ then $\bar{v}=0$ on $\gamma_0$; in particular, $\bar{v}\in H^{1}(0,T,H^1_0(\omega))$ if $v=0$ on $\gamma\times[-1,1]$.
    \end{enumerate}
    \end{theorem}
    
    \begin{remark} 
    This theorem is an extension of the parts $(a)$ and $(b)$ in Theorem 4.2-1, \cite{Ciarlet4b} and its proof follows straightforward from the result presented there. The main difference is that we are interested in obtaining the corresponding conclusions in the Bochner spaces. Therefore, most of the changes of the proof consist in adding an additional integral with respect to the time variable and proving the statements for the functions and their time derivatives, alternately, over the spaces $\LLO,$ $\LLo,$ $\LHO,$ $\LHo$.
    \end{remark}

 Next, we introduce two theorems that can be also seen as extensions of the Theorems 5.2-1  and 5.2-2, \cite{Ciarlet4b} defined over the corresponding Bochner spaces. Therefore, their proof follow similar arguments used in the results available. Firstly, let us define for each $\bv\in\WHOt$ the functions $\gab(\bv)\in \WLO$, $\rab(\bv)\in\WHMO$ and $\eab^1(\var;\bv)\in\WLO$ defined by
 \begin{align}\label{gab_g}
  \gab(\bv)&:= \frac{1}{2}(\d_\beta v_\alpha + \d_\alpha v_\beta) - \Gamma_{\alpha\beta}^\sigma v_\sigma -  b_{\alpha\beta}v_3,
  \\ \nonumber
  \rho_{\alpha\beta}(\bv)&:= \d_{\alpha\beta}v_3 - \Gamma_{\alpha\beta}^\sigma \d_\sigma v_3 - b_\alpha^\sigma b_{\sigma\beta} v_3 
  \\ \label{rab_g}
  & \qquad + b_\alpha^\sigma (\d_\beta v_\sigma- \Gamma_{\beta\sigma}^\tau v_\tau) + b_\beta^\tau(\d_\alpha v_\tau-\Gamma_{\alpha\tau}^\sigma v_\sigma ) + b^\tau_{\beta|\alpha} v_\tau,
  \\\label{eab2}
  \eab^1(\var;\bv)&:=\frac{1}{\var}\gab(\bv) + x_3(b_{\beta|\alpha}^\sigma v_\sigma  + b_\alpha^\sigma b_{\sigma\beta}v_3),
 \end{align}
 where the functions  $b_{\beta|\alpha}^\sigma$ are given by the expression introduce in (\ref{b_barra}).

 \begin{theorem}\label{Th_521}
  Let us identify  $\Gamma^\sigma_{\alpha\beta}, b_{\alpha\beta}, b_\alpha^\beta\in \mathcal{C}^0(\bar{\omega})$ with functions in $\mathcal{C}^0(\bar{\Omega})$ and let us consider  $\var_0$ defined in the Theorem  \ref{var_0}. Then, there exists a constant  $\tilde{C}>0$ such that for all  $\var$, $0<\var\leq \var_0$ and all $\bv\in H^1(0,T; [H^1(\Omega)]^3)$, the scaled linearized strains $\eab(\var;\bv)$ (see (\ref{eab})) satisfy
  \begin{align*}
  \left\|\frac{1}{\var}\eab(\var;\bv) - \eab^1(\var;\bv) \right\|_{\WLO}&\leq \tilde{C} \var \sum_{\alpha}|v_\alpha|_{\WLO},
  \\
  \left\|\frac{1}{\var}\d_3\eab(\var;\bv)+ \rab(\bv)\right\|_{\WHMO}&\leq \tilde{C}\left(\sum_i|e_{i||3}(\var;\bv)|_{\WLO} \right.
  \\
  \quad + \var\sum_\alpha|v_\alpha|_{\WLO} &+ \left. \var||v_3||_{\WHO}  \right).
  \end{align*}
  \end{theorem}   
  
  \begin{theorem}\label{Th_522}
Let $(\bu(\var))_{\var>0}$ be a sequence of functions $\bu(\var)\in\WVO$ such that
  \begin{align*}
  \bu(\var)&\deb\bu \en \WHOt,
  \\
  \frac{1}{\var}\eij(\var)&\deb \eij^1 \en \WLO,
  \end{align*}
  when $\var\to0.$ Then,
  \begin{enumerate}[label=(\alph*), leftmargin=0em ,itemindent=3em]
  \item   $\bu$ is independent of the transversal variable  $x_3$.
  \item   $\bar{\bu}\in H^1(\omega)\times H^1(\omega) \times H^2(\omega)$ with $\bar{u}_i=\d_\nu \bar{u}_3=0$ in $\gamma_0$ ($\d_\nu$ denotes the outer normal derivative operator along $\gamma$).
  \item  $\gab(\bu)=0$.
  \item  $\rab(\bu)\in \WLO$ and $\rab(\bu)=-\d_3\eab^1.$
  
  \item Moreover, if there exist functions $\kappa_{\alpha\beta}\in\WHMO$ such that $\rab(\bu(\var))\to\kappa_{\alpha\beta}$ in $\WHMO$ when $\var\to 0$, then
  \begin{align*}
  \bu(\var)&\to \bu \en \WHOt,\\
  \rab(\bu)&=\kappa_{\alpha\beta} \ \textrm{hence,} \ \kappa_{\alpha\beta}\in\WLO.
  \end{align*}
  \end{enumerate}
  \end{theorem}

 Finally, in the next theorem we recall a three-dimensional inequality of Korn's type for a family of viscoelastic shells, that can also be found in Theorem 5.3-1, \cite{Ciarlet4b}.

     \begin{theorem}\label{Th_desigKorn}
       Assume that $\btheta \in \mathcal{C}^3(\bar{\omega};\mathbb{R}^3)$ and we consider $\var_0$ defined as in Theorem \ref{var_0}. We consider a family of  shells with thickness $2\varepsilon$ with each having the same  middle surface $S=\btheta(\bar{\omega})$ and with each subjected to a boundary condition of place along a portion of its lateral face having the same set $\btheta(\gamma_0)$ as its middle curve.
         Then there exist a constant $\varepsilon_1$ verifying $0<\varepsilon_1<\varepsilon_0$ and a constant $C>0$ such that, for all $\var$, $0<\varepsilon\leq \varepsilon_1$, the following three-dimensional inequality of Korn's type holds,
         \begin{equation}\label{Korn}
       \left\| \bv\right\| _{1,\Omega}\leq \frac{C}{\var}\left(\sum_{i,j}|\eij(\varepsilon;\bv)|^2_{0,\Omega}\right)^{1/2} \ \forall \bv=(v_i)\in V(\Omega).
         \end{equation}
     \end{theorem}

\section{Asymptotic Analysis. Convergence results as $\var \to 0$}\setcounter{equation}{0} \label{seccion_convergencia}
Firstly, we  recall the two-dimensional equations obtained for a viscoelastic flexural shell as a consequence of the formal asymptotic study made in \cite{intro2}. 

 From the asymptotic analysis made in \cite{intro2}, we show that, if the applied body force density is $O(\var^2)$  and surface tractions density is $O(\var^3)$ in the Problem \ref{problema_orden_fuerzas}, we obtain the two-dimensional variational problem for a viscoelastic flexural shell. Let us remind the definition of the two-dimensional fourth-order tensors that appeared naturally in that study,
\begin{align} \label{tensor_a_bidimensional}
   \a&:=\frac{2\lambda\rho^2 + 4\mu\theta^2}{(\theta + \rho)^2}a^{\alpha\beta}a^{\sigma\tau} + 2\mu\ten, 
  \\ \label{tensor_b_bidimensional}
    \b&:=\frac{2\theta\rho}{\theta + \rho}a^{\alpha\beta}a^{\sigma\tau} + \rho\ten,
     \\ \label{tensor_c_bidimensional}
   \c&:=\frac{2 \left(\theta \Lambda \right)^2}{\theta + \rho} a^{\alpha\beta}a^{\sigma\tau},  
 \end{align}
 where
 \begin{align}\label{Lambda}
 \Lambda:=\left(   \frac{\lambda}{\theta}  - \frac{\lambda+ 2 \mu}{\theta + \rho} \right). 
 \end{align}
 Let
  $
   \rho_{\alpha\beta}(\beeta):= \d_{\alpha\beta}\eta_3 - \Gamma_{\alpha\beta}^\sigma \d_\sigma\eta_3 - b_\alpha^\sigma b_{\sigma\beta} \eta_3 + b_\alpha^\sigma (\d_\beta\eta_\sigma- \Gamma_{\beta\sigma}^\tau \eta_\tau) + b_\beta^\tau(\d_\alpha\eta_\tau-\Gamma_{\alpha\tau}^\sigma\eta_\sigma ) + b^\tau_{\beta|\alpha} \eta_\tau
 $
 and
 $\gab(\beeta):= \frac{1}{2}(\d_\beta\eta_\alpha + \d_\alpha \eta_\beta) - \Gamma_{\alpha\beta}^\sigma \eta_\sigma -  b_{\alpha\beta}\eta_3$,
  respectively
   denote the covariant components of the linearized change of curvature  and linearized change of metric tensors, both associated with a displacement  field $\eta_i \ba^i$ of the surface $S$. In what follows, we assume that the space of inextensional displacements, defined by
   \begin{align*}
    V_F(\omega):= \{ \beeta=(\eta_i) \in H^1(\omega)\times H^1(\omega)\times H^2(\omega); \eta_i=\d_\nu \eta_3=0 \ \textrm{on} \ \gamma_0, \gab(\beeta)=0 \en \omega \}, 
   \end{align*}    
 contains non-trivial functions. Therefore, we can enunciate the two-dimensional variational problem for a linear viscoelastic flexural shell:
\begin{problem}\label{problema_ab}
 Find $\bxi:[0,T] \times\omega \longrightarrow \mathbb{R}^3$ such that,
    \begin{align}\nonumber 
       & \bxi(t,\cdot)\in V_F(\omega) \forallt,\\ \nonumber
      &\frac{1}{3}\int_{\omega} \a\rst(\bxi)\rab(\beeta)\sqrt{a}dy +\frac{1}{3}\int_{\omega}\b\rst(\dot{\bxi})\rab(\beeta)\sqrt{a}dy
      \\  \nonumber
      &- \frac{1}{3}\int_0^te^{-k(t-s)}\into \c \rst(\bxi(s))\rab(\beeta)\sqrt{a}dyds
      \\  \nonumber
      & =\int_{\omega}p^{i}\eta_i\sqrt{a}dy \quad \forall \beeta=(\eta_i)\in V_F(\omega), \ae
       \\ \nonumber
       &\bxi(0,\cdot)=\bxi_0(\cdot),
      \end{align} 
   where we introduced the constant $k$ defined by
   \begin{align}\label{k}
   k:=\frac{\lambda+ 2 \mu}{\theta + \rho},
   \end{align}
   and
   \begin{align}\label{def_p}
    p^{i}(t):=\int_{-1}^{1}\fb^{i,2}(t)dx_3+h_+^{i,3}(t)+h_-^{i,3}(t), \ \textrm{with} \ h_{\pm}^{i,3}(t)=\bh^{i,3}(t,\cdot,\pm 1).
   \end{align}
\end{problem}
The Problem \ref{problema_ab} is well posed and it has existence and uniqueness of solution. Furthermore, we obtained the following result (see \cite{intro2} for details of the proof of the de-scaled version):

\begin{theorem} \label{Th_exist_unic_bid_cero}
Let $\omega$  be a domain in $\mathbb{R}^2$, let $\btheta\in\mathcal{C}^2(\bar{\omega};\mathbb{R}^3)$ be an injective mapping such that the two vectors $\ba_\alpha=\d_\alpha\btheta$ are linearly independent at all points of $\bar{\omega}$. Let $\fb^{i,2}\in L^{2}(0,T; L^2(\Omega)) $, $\bh^{i,3}\in L^{2}(0,T; L^2(\Gamma_1))$, where $\Gamma_1:= \Gamma_+\cup\Gamma_-$. Let  $\bxi_0\in V_F(\omega). $  Then the Problem \ref{problema_ab}, has a unique solution  $\bxi\in H^{1}(0,T;V_F(\omega))$.  In addition to that, if $\dot{\fb}^{i,2}\in L^{2}(0,T; L^2(\Omega)) $, $\dot{\bh}^{i,3}\in L^{2}(0,T; L^2(\Gamma_1))$, then $\bxi\in H^{2}(0,T;V_F(\omega))$. 
\end{theorem}

 For each $\var>0$,  we assume that the initial condition for the scaled linear strains is
\begin{equation} \label{condicion_inicial_def}
\eij(\var)(0,\cdot)=0,
\end{equation}
this is,  the domain is on its natural state with no strains on it at the beginning of the period of observation.

Now, we present here the main result of this paper, namely that the scaled three-dimensional unknown $\bu(\var)$ converges, as $\var$ tends to zero, towards a limit $\bu$ independent of the transversal variable. Moreover, this limit can be identified with the solution $\bxi=\bar{\bu}$ of the Problem \ref{problema_ab}, posed over the set $\omega$.

   \begin{theorem}\label{Th_convergencia}
   Assume that $\btheta\in\mathcal{C}^3(\bar{\omega};\mathbb{R}^3)$. Consider a family of viscoelastic flexural shells with thickness $2\var$ approaching zero and with each having the same middle surface $S=\btheta(\bar{\omega})$, and let the assumptions on the data be as in Theorem \ref{Th_exist_unic_bid_cero}.
   	 For all $\var$, $0<\varepsilon\leq\varepsilon_0$
   let $\bu(\varepsilon)$ be  the solution of the associated three-dimensional scaled Problem \ref{problema_orden_fuerzas}. Then, there exists a function $\bu\in \WHOt$ satisfying $\bu=\bcero$ on $\Gamma_0:=\gamma_0\times[-1,1]$  such that
   \begin{enumerate}[label={{(\alph*)}}, leftmargin=0em ,itemindent=3em]
   \item $\bu(\varepsilon)\rightarrow \bu$ in $\WHOt$  when $\varepsilon\rightarrow 0$,
   \item $\bu:=(u_i)$ is independent of the transversal variable $x_3$.
   \end{enumerate}
   Furthermore, the average $ \displaystyle\bar{\bu}:= \frac1{2}\int_{-1}^{1}\bu dx_3$ verifies the Problem \ref{problema_ab}. 
   \end{theorem}
\begin{proof} 
We follow the same structure of the proof in
 Theorem 6.2-1, \cite{Ciarlet4b}. The proof is divided into several parts, numbered from $(i)$ to $(vi)$. Moreover, we will use the notation $f^i\equiv f^{i,2}$ and $h^i\equiv h^{i,3}$ , for  notational brevity.
  \begin{enumerate}[label={{(\roman*)}}, leftmargin=0em ,itemindent=3em]
   \item {\em \textit{A priori} boundedness and extraction of weak convergent sequences.
     
     The norms $|\frac{1}{\var}\eij(\varepsilon)|_{H^{1}(0, T, L^2(\Omega))}$ and $ ||\bu(\varepsilon)||\nWHOt$ are bounded independently of  $\var, 0<\varepsilon\leq\varepsilon_1, $ where $ \varepsilon_1>0$ is given in Theorem \ref{Th_desigKorn}.
     Consequently, there exists a subsequence, also denoted 
   $ (\bu(\varepsilon))_{\varepsilon>0}$, and functions  $\eij^1\in H^{1}(0, T, L^2(\Omega)),\bu\in \WHOt $, satisfying $\bu=\bcero$ on $\Gamma_0$ such that
     \begin{align}
     &\bu(\varepsilon)\deb\bu \en \WHOt \\ &\textrm{and hence} \ \bu(\var)\rightarrow \bu \en\WLOt,
      \\\label{eij_conv}
         & \frac{1}{\var}\eij(\varepsilon)\deb\eij^1\en \WLO,
     \end{align}}
     
    For the proof of this step we take $\bv=\bu(\var)\dep$ in (\ref{ecuacion_orden_fuerzas}) and find
     \begin{align}\nonumber
               &\int_{\Omega}A^{ijkl}(\varepsilon)e_{k||l}(\varepsilon)e_{i||j}(\varepsilon)\sqrt{g(\varepsilon)} dx
             \displaystyle + \int_{\Omega} B^{ijkl}(\varepsilon)\dot{e}_{k||l}(\varepsilon)e_{i||j}(\varepsilon) \sqrt{g(\varepsilon)}  dx
              \\ \nonumber
              &\quad = \var^2\int_{\Omega} f^{i} u_i(\var) \sqrt{g(\varepsilon)} dx +\var^2\int_{\Gamma_1} h^i u_i(\var)\sqrt{g(\var)} d\Gamma, \aes,  
             \end{align} 
  which is equivalent to,           
        \begin{align}\nonumber
                    &\int_{\Omega}A^{ijkl}(\varepsilon)e_{k||l}(\varepsilon)e_{i||j}(\varepsilon)\sqrt{g(\varepsilon)} dx
                  \displaystyle + \frac1{2} \frac{\d}{\d t}\int_{\Omega} B^{ijkl}(\varepsilon){e}_{k||l}(\varepsilon)e_{i||j}(\varepsilon) \sqrt{g(\varepsilon)}  dx
                   \\ \nonumber
                   &\quad = \var^2\int_{\Omega} f^{i} u_i(\var) \sqrt{g(\varepsilon)} dx +\var^2\int_{\Gamma_1} h^i u_i(\var)\sqrt{g(\var)} d\Gamma, \aes.  
                  \end{align}       
  Now, integrating over $[0,T]$ and using (\ref{elipticidadB_eps}) and (\ref{condicion_inicial_def}), we find that
    \begin{align} \nonumber
                &\int_0^T \left(\int_{\Omega} A^{ijkl}(\varepsilon){e}_{k||l}(\varepsilon)e_{i||j}(\varepsilon) \sqrt{g(\varepsilon)}  dx\right) dt
               \\ \label{desig_i}
               & \quad \leq  \var^2\int_0^T \left(\int_{\Omega} \fb^{i} u_i(\var) \sqrt{g(\varepsilon)} dx +\int_{\Gamma_1} h^i u_i(\var)\sqrt{g(\var)} d\Gamma \right)dt,  
               \end{align} 
   Using the Cauchy-Schwartz inequality and (\ref{g_acotado}), there exists a constant $\hat{c}>0$ depending on the norms $\left|f^i\right|_{0, \Omega}$, $\left|h^i\right|_{L^{2}(\Gamma^1)}$ and the norm of the trace operator from $H^1(\Omega)$ to $L^2(\Gamma^1)$, such that
     \begin{align}\label{desig1}
    \left| \int_{\Omega} f^i u_i(\var) \sqrt{g(\varepsilon)} dx +  \int_{\Gamma^1} h^i u_i(\var) \sqrt{g(\varepsilon)} d\Gamma \right|\leq \hat{c}\left\|\bu(\var)\right\|_{1,\Omega},
     \end{align}
   for all $\var$, $0<\var \leq \var_0$ and for all $t\in[0,T].$  
   On the other hand, by using (\ref{elipticidadA_eps}), (\ref{g_acotado}) and  (\ref{Korn}) we obtain   
  \begin{align}   \label{desig2}
      &\int_{\Omega} A^{ijkl}(\varepsilon){e}_{k||l}(\varepsilon)e_{i||j}(\varepsilon) \sqrt{g(\varepsilon)}  dx
  \geq g_0^{1/2} C_e |\ekl(\varepsilon)|_{0,\Omega}^2
 \geq g_0^{1/2}C_e C^{-2} \var^2\left\|\bu(\var)\right\|_{1,\Omega}^2.
      \end{align} 
  Now, (\ref{desig_i})--(\ref{desig2}) together  and the Cauchy-Schwartz inequality imply that 
   \begin{align} \nonumber
    &g_0^{1/2}C_e C^{-2} \var^2 \left\|\bu(\var)\right\|_{L^2\left(0,T;[H^1(\Omega)]^3\right)}^2
      \leq \hat{c} \var^2\int_0^T \left\|\bu(\var)\right\|_{1,\Omega} dt
 \leq\hat{c}\sqrt{T}  \var^2\left\|\bu(\var)\right\|_{L^2\left(0,T;[H^1(\Omega)]^3\right)}.
   \end{align}    
   Hence, we conclude that there exists a constant $\tilde{k}_1>0$ independent of $\varepsilon$ such that $\left\|\bu(\var)\right\|_{L^2\left(0,T;[H^1(\Omega)]^3)\right)}\leq\tilde{k}_1.$ As a consequence, by the inequalities  (\ref{desig_i})--(\ref{desig2}) we can check that the bounds for the norms $\left|\frac{1}{\var}\eij(\var)\right|_{L^2\left(0,T;H^1(\Omega)\right)}$ also hold.

 Next,  we take $\bv=\dot{\bu}(\var)\dep$ in (\ref{ecuacion_orden_fuerzas}) and find that
        \begin{align}\nonumber
                  &\int_{\Omega}A^{ijkl}(\varepsilon)e_{k||l}(\varepsilon)\deij(\varepsilon)\sqrt{g(\varepsilon)} dx
                \displaystyle + \int_{\Omega} B^{ijkl}(\varepsilon)\dekl(\varepsilon)\deij(\varepsilon) \sqrt{g(\varepsilon)}  dx
                 \\ \nonumber
                 &\quad = \var^2\int_{\Omega} f^{i} \dot{u}_i(\var) \sqrt{g(\varepsilon)} dx +\var^2\int_{\Gamma_1} h^i\dot{u}_i(\var)\sqrt{g(\var)} d\Gamma, \aes,  
                \end{align} 
     which is equivalent to,           
           \begin{align}\nonumber
                       &\frac1{2} \frac{\d}{\d t}\int_{\Omega}A^{ijkl}(\varepsilon)e_{k||l}(\varepsilon)e_{i||j}(\varepsilon)\sqrt{g(\varepsilon)} dx
                     \displaystyle +\int_{\Omega} B^{ijkl}(\varepsilon)\dekl(\varepsilon)\deij(\varepsilon) \sqrt{g(\varepsilon)}  dx
                      \\ \nonumber
                      &\quad = \var^2\int_{\Omega} f^{i} \dot{u}_i(\var) \sqrt{g(\varepsilon)} dx +\var^2\int_{\Gamma_1} h^i \dot{u}_i(\var)\sqrt{g(\var)} d\Gamma, \aes.  
                     \end{align}       
     Integrating over $[0,T]$ and using (\ref{elipticidadB_eps}) and (\ref{condicion_inicial_def}), we find that
       \begin{align} \nonumber
                   &\int_0^T \left(\int_{\Omega} B^{ijkl}(\varepsilon)\dekl(\varepsilon)\deij(\varepsilon) \sqrt{g(\varepsilon)}  dx\right) dt
                  \\ \label{desig_id}
                  & \quad \leq  \var^2\int_0^T \left(\int_{\Omega} \fb^{i} \dot{u}_i(\var) \sqrt{g(\varepsilon)} dx +\int_{\Gamma_1} h^i \dot{u}_i(\var)\sqrt{g(\var)} d\Gamma \right)dt,  
                  \end{align} 
      Using the Cauchy-Schwartz inequality and (\ref{g_acotado}), again there exists a constant $c>0$ depending on the norms$\left|f^i\right|_{0, \Omega}$, $\left|h^i\right|_{L^{2}(\Gamma^1)}$ and the norm of the trace operator from $H^1(\Omega)$ to $L^2(\Gamma^1)$, such that
        \begin{align}\label{desig1d}
       \left| \int_{\Omega} f^i \dot{u}_i(\var) \sqrt{g(\varepsilon)} dx +  \int_{\Gamma^1} h^i \dot{u}_i(\var) \sqrt{g(\varepsilon)} d\Gamma \right|\leq c\left\|\dot{\bu}(\var)\right\|_{1,\Omega},
        \end{align}
      for all $\var$, $0<\var \leq \var_0$ and for all $t\in[0,T].$  
      On the other hand, by using (\ref{elipticidadB_eps}), (\ref{g_acotado}) and  (\ref{Korn}) we obtain   
     \begin{align}  \label{desig2d}
         &\int_{\Omega} B^{ijkl}(\varepsilon)\dekl(\varepsilon)\deij(\varepsilon) \sqrt{g(\varepsilon)}  dx
     \geq g_0^{1/2} C_v |\dekl(\varepsilon)|_{0,\Omega}^2
\geq g_0^{1/2}C_v C^{-2} \var^2\left\|\dot{\bu}(\var)\right\|_{1,\Omega}^2.
         \end{align} 
     Therefore, (\ref{desig_id})--(\ref{desig2d}) together  and the Cauchy-Schwartz inequality imply that 
      \begin{align} \nonumber
       &g_0^{1/2}C_v C^{-2} \var^2 \left\|\dot{\bu}(\var)\right\|_{L^2\left(0,T;[H^1(\Omega)]^3\right)}^2
         \leq c \var^2\int_0^T \left\|\dot{\bu}(\var)\right\|_{1,\Omega} dt
    \leq c\sqrt{T} \var^2\left\|\dot{\bu}(\var)\right\|_{L^2\left(0,T;[H^1(\Omega)]^3\right)}.
      \end{align}    
      Hence, we conclude that there exists a constant $\tilde{k}_2>0$ independent of $\varepsilon$ such that $\left\|\dot{\bu}(\var)\right\|_{L^2\left(0,T;[H^1(\Omega)]^3)\right)}\leq\tilde{k}_2$. As a consequence, by the inequalities  (\ref{desig_id})--(\ref{desig2d}) we can check that the bounds for the norms $\left|\frac{1}{\var}\deij(\var)\right|_{L^2\left(0,T;H^1(\Omega)\right)}$ also hold.  Therefore, the convergences announced in this step are satisfied.
      
 \item   {\em The limit $\bu$ found in the previous step  is independent fo the transversal variable $x_3$ and its average $\frac{1}{2}\int_{-1}^1\bu dx_3=\bar{\bu}\in H^1(\omega)\times H^1(\omega)\times H^2(\omega)$  and $\bar{u}_i=\d_{\nu}\bar{u}_3=0$ at $\gamma_0$ and $\gab(\bar{\bu})=0$ in $\omega$ for all $t\in[0,T].$
 }
 
 This step is consequence of the step $(i)$ and the Theorem \ref{Th_522}. Hence, the proof of the step $(b)$ of this theorem is met.
 
 \item {\em We obtain the relation between the limits $\eij^1$ found in $(i)$ and the limits   $\bu:=(u_i)$. }  
 Firstly, by the Theorem \ref{Th_522} we have that $-\d_3\eab^1=\rab(\bu)$ in $\WLO$.  
  As a consequence of the definition of the scaled strains in (\ref{eab})--(\ref{edtres}), we find
 \begin{align}\nonumber
 &\varepsilon\eab(\varepsilon;\bv)\rightarrow 0 \ \textrm{in} \ L^2(\Omega),
 \\\nonumber
 &\varepsilon\eatres(\varepsilon;\bv)\rightarrow \frac1{2}\d_3v_\alpha \  \textrm{in} \ L^2(\Omega),
 \\\nonumber
&\varepsilon\edtres(\varepsilon;\bv)=\d_3v_3 \ \textrm{for all} \ \varepsilon>0.
 \end{align}
   Using the variational formulation (\ref{ecuacion_orden_fuerzas})  and taking into account (\ref{tensor_terminos_nulos}), (\ref{tensorA_escalado}) and (\ref{tensorB_escalado}), we have
  \begin{align}\nonumber
           &\int_{\Omega}A^{ijkl}(\varepsilon) e_{k||l}(\varepsilon)e_{i||j}(\varepsilon,\bv)\sqrt{g(\varepsilon)} dx + \int_{\Omega} B^{ijkl}(\varepsilon)\dot{e}_{k||l}(\varepsilon)e_{i||j}(\varepsilon,\bv) \sqrt{g(\varepsilon)}  dx
          \\ \nonumber
          &\quad=\frac{1}{\var}\int_{\Omega}\left(A^{\alpha\beta\sigma\tau}(\varepsilon)\est(\varepsilon) + A^{\alpha\beta33}(\varepsilon)\edtres(\varepsilon)\right)(\varepsilon\eab(\varepsilon;\bv))\sqrt{g(\varepsilon)}dx 
           \\ \nonumber
          &\qquad +\frac{1}{\var}\int_{\Omega}4A^{\alpha3\sigma3}(\varepsilon)\estres(\varepsilon)\left(\varepsilon\eatres(\varepsilon;\bv)\right)\sqrt{g(\varepsilon)}dx
          \\ \nonumber
          &\qquad +\frac{1}{\var}\int_{\Omega}\left(A^{33\sigma\tau}(\varepsilon)\est(\varepsilon) + A^{3333}(\varepsilon)\edtres(\varepsilon)\right)(\varepsilon\edtres(\varepsilon;\bv))\sqrt{g(\varepsilon)}dx  
          \\\nonumber
          &\qquad
           +\frac{1}{\var}\int_{\Omega}\left(B^{\alpha\beta\sigma\tau}(\varepsilon)\dest(\varepsilon) + B^{\alpha\beta33}(\varepsilon)\dedtres(\varepsilon)\right)(\varepsilon\eab(\varepsilon;\bv))\sqrt{g(\varepsilon)}dx
        \\ \nonumber
          &\qquad
           +\frac{1}{\var}\int_{\Omega}4B^{\alpha3\sigma3}(\varepsilon)\destres(\varepsilon)\left(\varepsilon\eatres(\varepsilon;\bv)\right)\sqrt{g(\varepsilon)}dx
    \\\qquad \nonumber
          &\qquad+ \frac{1}{\var}\int_{\Omega}\left(B^{33\sigma\tau}(\varepsilon)\dest(\varepsilon) + B^{3333}(\varepsilon)\dedtres(\varepsilon)\right)(\varepsilon\edtres(\varepsilon;\bv))\sqrt{g(\varepsilon)}dx
          \\\nonumber
          &\quad= \varepsilon^2\int_{\Omega} f^{i} v_i \sqrt{g(\varepsilon)} dx + \varepsilon^2\int_{\Gamma_1}  h^{i} v_i \sqrt{g(\varepsilon)} d\Gamma\quad \forall \bv\in V(\Omega), \aes.
         \end{align} 
We pass to the limit as $\var\to0$ and by taking into account the asymptotic behavior of the contravariant components of the fourth order tensors  $A^{ijkl}(\varepsilon)$, $B^{ijkl}(\varepsilon)$ (see Theorem \ref{Th_comportamiento asintotico}), $g(\varepsilon)$ (see Theorem \ref{Th_simbolos2D_3D}), the convergences above and the weak convergences of the step $(i)$ we obtain the following integral equation
 \begin{align}\nonumber 
 &\int_{\Omega}\left(2\mu a^{\alpha\sigma}\eatres^1\d_3v_\sigma + (\lambda + 2\mu)\edtres^1\d_3v_3\right)\sqrt{a}dx
 +\int_{\Omega}\lambda a^{\alpha\beta}\eab^1\d_3v_3\sqrt{a}dx 
 \\ \nonumber
 & \quad
  +\int_{\Omega}\left(\rho a^{\alpha\sigma}\deatres^1\d_3v_\sigma + (\theta + \rho){\dedtres^1}\d_3v_3\right)\sqrt{a} dx
  \\\label{ecuacion_integral1}
  & \quad
 +\int_{\Omega}\theta a^{\alpha\beta}{\deab^1}\d_3v_3\sqrt{a}dx=0,
 \end{align}
 in $\Omega$, $\aes$. On one hand, if we take $\bv\in V(\Omega)$ such that $v_2=v_3=0$ and using the Theorem \ref{th_int_nula}, we have
 \begin{align} \label{ec_dif1}
 2\mu a^{\alpha 1} \eatres^1 + \rho a^{\alpha 1} \deatres^1 =0, \aes.
 \end{align}
 On the other hand, if we take $\bv\in V(\Omega)$ such that $v_1=v_3=0$ and using the Theorem \ref{th_int_nula}, we have
 \begin{align} \label{ec_dif2}
 2\mu a^{\alpha 2} \eatres^1 + \rho a^{\alpha 2} \deatres^1 =0, \aes.
 \end{align}
 
 Multiplying (\ref{ec_dif1}) by $a^{22}$ and (\ref{ec_dif2}) by $-a^{21}$ and adding both expressions we have
 \begin{align*}
 2\mu \left(a^{22}a^{11} - a^{21}a^{12} \right) e_{1||3}^1 + \rho \left( a^{22}a^{11} - a^{21}a^{12}\right)\dot{e}_{1||3}^1 = 2\mu ae_{1||3}^1 + \rho a\dot{e}_{1||3}^1 =  0,
 \end{align*}
 $\aes$, by (\ref{definicion_a}). Now, by the initial condition in (\ref{condicion_inicial_def}) we conclude
 \begin{align*}
 e_{1||3}^1(t) =0 \en \Omega ,\ \textrm{for all} \ t\in(0,T).
 \end{align*}
 Multiplying (\ref{ec_dif1}) by $a^{12}$ and (\ref{ec_dif2}) by $-a^{11}$ and adding both expressions we have
 \begin{align*}
 2\mu ae_{2||3}^1  + \rho a\dot{e}_{2||3}^1 =  0, \aes.
 \end{align*}
  Now, by the initial condition in (\ref{condicion_inicial_def}) we conclude
 \begin{align*}
 e_{2||3}^1(t) =0 \en \Omega ,\ \textrm{for all} \ t\in(0,T).
 \end{align*}
 Taking  $\bv\in V(\Omega)$ such that $v_\alpha=0$ in (\ref{ecuacion_integral1}) , we obtain
 \begin{align} \nonumber
 &\intO  \left(\lambda a^{\alpha\beta}\eab^1+ (\lambda + 2\mu)\edtres^1 \right)\d_3v_3\sqrt{a}dx
 + \intO  \left(\theta a^{\alpha\beta}\deab^1+ (\theta + \rho)\dedtres^1 \right)\d_3v_3\sqrt{a}dx=0,
 \end{align}
 for all $v_3\in H^1(\Omega)$ with $v_3=0 \en \Gamma_0$ and $\aes$. By Theorem \ref{th_int_nula}, we obtain the following differential equation
 \begin{align} \label{ecuacion_casuistica}
 \lambda a^{\alpha\beta} \eab^1 + (\lambda+2\mu) \edtres^1 +\theta a^{\alpha \beta} \deab^1 + (\theta + \rho) \dedtres^1=0.
 \end{align}
 \begin{remark} \label{nota_desvio}
 Note that removing time dependency and viscosity, that is taking $\theta=\rho=0$, the equation leads to the one studied  \cite{Ciarlet4b} (page 312, Chapter 6), that is, the elastic case. 
 \end{remark}
 In order to solve the equation (\ref{ecuacion_casuistica}) in the more general case, we assume that the viscosity coefficient $\theta$ is strictly positive. Thus, we can prove that this equation is equivalent to
 \begin{align} \nonumber 
 \theta e^{-\frac{\lambda}{\theta}t} \frac{\d}{\d t}\left(a^{\alpha\beta}\eab^1 e^{\frac{\lambda}{\theta}t}\right)=-\left(\theta + \rho \right)e^{-\frac{\lambda + 2\mu}{\theta + \rho}t} \frac{\d}{\d t}\left(\edtres^1 e^{\frac{\lambda + 2\mu}{\theta + \rho}t}\right).
 \end{align}
 Integrating with respect to the time variable and using (\ref{condicion_inicial_def}) we find that,
 \begin{align*}
 \edtres^1 e^{\frac{\lambda + 2\mu}{\theta + \rho}t}= - \frac{\theta}{\theta + \rho} \int_0^t e^{\left(\frac{\lambda +2 \mu}{\theta + \rho} - \frac{\lambda}{\theta}\right) s}  \frac{\d}{\d s} \left(a^{\alpha\beta}\eab^1(s) e^{\frac{\lambda}{\theta}s}\right) ds,
 \end{align*}
 now integrating by parts and simplifying we conclude that,
 \begin{align}\label{edtres1}
 \edtres^1(t)= - \frac{\theta}{\theta + \rho} \left( a^{\alpha \beta }\eab^1(t) + \Lambda\int_0^te^{-k(t-s)}a^{\alpha\beta}\eab^1(s) ds \right),
 \end{align}
 in $\Omega$ , $\forallt$, and where $\Lambda$ and $k$ are defined  in (\ref{Lambda}) and (\ref{k}), respectively. Moreover, from (\ref{ecuacion_casuistica}) we obtain that,
 \begin{align}\label{dedtres1}
 \dedtres^1(t)= - \frac{\lambda}{\theta + \rho} a^{\alpha \beta} \eab^1(t)- \frac{\lambda + 2\mu}{\theta + \rho} \edtres^1(t) - \frac{\theta}{\theta + \rho}  a^{\alpha\beta}\deab^1(t).
 \end{align}
 in $\Omega$ , $\ae.$

 \item {\em The function $\bar{\bu}:=\left(\bar{u}_i\right)$ verifies the Problem \ref{problema_ab} which has uniqueness of solution by the Theorem \ref{Th_exist_unic_bid_cero}. As a consequence, the convergences $\bu(\var)\deb\bu$ in $\WHOt$ and $\bu(\var)\to\bu$ in $\WLOt$ are verified by the whole family $\left(\bu(\var)\right)_{\var>0}$ (if the function $\bar{\bu}$ is unique, so is the function $\bu$ as it is independent of $x_3$ by the step $(ii)$).}
 
 The function $\bar{\bu}$ belongs to the space $V_F(\omega),$ for all $t\in[0,T]$ by the step $(ii).$ Given $\beeta\in V_F(\omega),$ let $\bv(\var)=\left(v_i(\var)\right)$ be defined almost everywhere in $\Omega$ for all $\var>0$ as follows (as in the elastic case we follow the idea taken from \cite{MiaraSPalencia1996}):
 \begin{align}\label{defv_MP1}
 v_\alpha(\var)&:= \eta_\alpha - \var x_3\theta_\alpha, \textrm{ with } \theta_\alpha:=\d_\alpha\eta_3 + 2b_\alpha^\sigma \eta_\sigma.
 \\\label{defv_MP2}
 v_3(\var)&:= \eta_3.
 \end{align}
 Then, we have that $\bv(\var)\in V(\Omega)$ and $\edtres(\var;\bv(\var))=0$  for all $\var>0$. Let us prove that for a  function $\beeta\in V_F(\omega)$ identified wherever is needed with a function in the space $H^1(\Omega)\times H^1(\Omega)\times H^2(\Omega)$ we obtain the following:
 \begin{align}\label{ab_1}
 \bv(\var)&\to \beeta \in [H^1(\Omega)]^3,
 \\\label{ab_2}
 \frac{1}{\var}\eab (\var;\bv(\var))&\to \left(-x_3 \rab(\beeta)\right) \in L^2(\Omega), \textrm{ when } \var\to 0,
 \\\label{ab_3}
 \left(\frac{1}{\var}\eatres\left(\var;\bv(\var)\right)\right)_{\var>0} &\textrm{ converges in } L^2(\Omega).
  \end{align}
 The first relation holds by the definition od the function $\bv(\var)$ in (\ref{defv_MP1})--(\ref{defv_MP2}). Using the fact that $\gab(\beeta)=0$ (since $\beeta\in V_F(\omega)$) we obtain after some calculations that
 \begin{align*}
 \eab^1(\var;\bv(\var)) :=& \frac{1}{\var} \gab(\bv(\var)) + x_3b_{\beta|\alpha}^\sigma v_\sigma(\var) + x_3b_\alpha^\sigma b_{\sigma\beta}v_3(\var)
 \\
 =&- x_3 \left(\frac{1}{2}\left(\d_\beta\theta_\alpha + \d_\alpha\theta_\beta\right) - \Gamma_{\alpha\beta}^\sigma\theta_\sigma - b_{\beta|\alpha}^\sigma\eta_\sigma - b_\alpha^\sigma b_{\sigma\beta}\eta_3\right) - \var x_3^2 b_{\beta|\alpha}^\sigma\theta_\sigma
 \\
 =& - x_3 \rab(\beeta)  - \var x_3^2 b_{\beta|\alpha}^\sigma\theta_\sigma.
 \end{align*}   
Therefore, by the applying Theorem \ref{Th_521} we have that
\begin{align*}
\left|\frac{1}{\var}\eab(\var;\bv(\var)) - \eab^1(\var;\bv(\var))\right|_{0,\Omega}\leq c_1\var \sum_{\alpha} \left|v_\alpha(\var)\right|_{0,\Omega},
\end{align*}     
  and thus,
  \begin{align*}
  \frac{1}{\var}\eab(\var;\bv(\var)) \to - x_3 \rab(\beeta)  \in L^2(\Omega).
  \end{align*}   
 Also, we can obtain after some calculations that 
 \begin{align*}
 \frac{1}{\var}\eatres \left(\var;\bv(\var)\right)= -\frac{1}{\var}\left(\Gamma^\sigma_{\alpha3}(\var) + b_\alpha^\sigma\right)\eta_\sigma + x_3\Gamma^\sigma_{\alpha 3}(\var) \theta_\sigma,
 \end{align*}
 which together with the asymptotic behaviour of the functions $\Gamma^\sigma_{\alpha3}(\var)$  (see Theorem \ref{Th_simbolos2D_3D}), imply that $\left(\frac{1}{\var}\eatres\left(\var;\bv(\var)\right)\right)_{\var>0}$ converges in $L^2(\Omega)$. 
 
 Now, let  $\beeta\in V_F(\omega)$ fixed, and take $\bv=\bv(\var)$ in the equation (\ref{ecuacion_orden_fuerzas}), with $\bv(\var)=(v_i(\var))$ defined as in (\ref{defv_MP1})--(\ref{defv_MP2}) to obtain that 
  \begin{align}\nonumber
            &\lim_{\var\to0} \left(\frac{1}{\var^2}\int_{\Omega}A^{ijkl}(\varepsilon) e_{k||l}(\varepsilon)e_{i||j}(\varepsilon;\bv)\sqrt{g(\varepsilon)} dx + \frac{1}{\var^2}\int_{\Omega} B^{ijkl}(\varepsilon)\dot{e}_{k||l}(\varepsilon)e_{i||j}(\varepsilon;\bv) \sqrt{g(\varepsilon)}  dx\right)
           \\ \nonumber
           &\quad=\lim_{\var\to0}\Bigg(\int_{\Omega}\left(A^{\alpha\beta\sigma\tau}(\varepsilon)\left(\frac{1}{\var}\est(\varepsilon)\right) + A^{\alpha\beta33}(\varepsilon)\left(\frac{1}{\var}\edtres(\varepsilon)\right)\right)\left(\frac{1}{\var}\eab(\varepsilon;\bv)\right)\sqrt{g(\varepsilon)}dx 
            \\ \nonumber
           &\qquad +\int_{\Omega}4A^{\alpha3\sigma3}(\varepsilon)\left(\frac{1}{\var}\estres(\varepsilon)\right)\left(\frac{1}{\var}\eatres(\varepsilon;\bv)\right)\sqrt{g(\varepsilon)}dx
           \\ \nonumber
           &\qquad +\int_{\Omega}\left(A^{33\sigma\tau}(\varepsilon)\left(\frac{1}{\var}\est(\varepsilon)\right) + A^{3333}(\varepsilon)\left(\frac{1}{\var}\edtres(\varepsilon)\right)\right)\left(\frac{1}{\var}\edtres(\varepsilon;\bv)\right)\sqrt{g(\varepsilon)}dx  
           \\\nonumber
           &\qquad
            +\int_{\Omega}\left(B^{\alpha\beta\sigma\tau}(\varepsilon)\left(\frac{1}{\var}\dest(\varepsilon)\right) + B^{\alpha\beta33}(\varepsilon)\left(\frac{1}{\var}\dedtres(\varepsilon)\right)\right)\left(\frac{1}{\var}\eab(\varepsilon;\bv)\right)\sqrt{g(\varepsilon)}dx
         \\ \nonumber
           &\qquad
            +\int_{\Omega}4B^{\alpha3\sigma3}(\varepsilon)\left(\frac{1}{\var}\destres(\varepsilon)\right)\left(\frac{1}{\var}\eatres(\varepsilon;\bv)\right)\sqrt{g(\varepsilon)}dx
     \\\qquad \nonumber
           &\qquad+ \int_{\Omega}\left(B^{33\sigma\tau}(\varepsilon)\left(\frac{1}{\var}\dest(\varepsilon)\right) + B^{3333}(\varepsilon)\left(\frac{1}{\var}\dedtres(\varepsilon)\right)\right)\left(\frac{1}{\var}\edtres(\varepsilon;\bv)\right)\sqrt{g(\varepsilon)}dx\Bigg)
           \\\nonumber
           &\quad=\lim_{\var\to0}\left(\int_{\Omega} f^{i} v_i \sqrt{g(\varepsilon)} dx + \int_{\Gamma_1}  h^{i} v_i \sqrt{g(\varepsilon)} d\Gamma\right).
          \end{align}  
 Let $\var\to0$. By the asymptotic behaviour of the functions $\bv(\var)$ and $ \frac{1}{\var}\eab (\var;\bv(\var))$ in (\ref{ab_1})--(\ref{ab_3}), $A^{ijkl}(\var)$, $B^{ijkl}(\var)$, $g(\var)$ (see Theorems \ref{Th_comportamiento asintotico}) and \ref{Th_simbolos2D_3D}), the weak convergences from the step $(i)$ and that $\eatres^1=0$ (see step $(iii)$) we find that
 \begin{align*}
&\intO A^{\alpha\beta\sigma \tau}(0)\est^1\left(-x_3 \rab(\beeta)\right) \sqrt{a}dx  
 + \intO A^{\alpha\beta33}(0)\edtres^1\left(-x_3 \rab(\beeta)\right) \sqrt{a}dx
 \\
 &\qquad+\intO B^{\alpha\beta\sigma \tau}(0)\dest^1\left(-x_3 \rab(\beeta)\right) \sqrt{a}dx  
  + \intO B^{\alpha\beta33}(0)\dedtres^1\left(-x_3 \rab(\beeta)\right) \sqrt{a}dx
  \\\nonumber
        &\quad=\left(\int_{\Omega} f^{i} \eta_i \sqrt{a} dx + \int_{\Gamma_1}  h^{i} \eta_i \sqrt{a} d\Gamma\right).
 \end{align*}
  Then, using the relations (\ref{edtres1})--(\ref{dedtres1}) from the step $(iii)$, we obtain that,
 \begin{align*}
 &\intO \left( \left(\lambda - \frac{\theta}{\theta + \rho} \left( \theta \Lambda + \lambda \right)\right)a^{\alpha\beta}a^{\sigma \tau} + \mu\ten  \right) \est^1\left(-x_3 \rab(\beeta)\right)\sqrt{a}dx
 \\
 &\qquad + \intO \left( \frac{\theta \rho }{\theta + \rho} a^{\alpha \beta} a^{\sigma \tau}  + \frac{\rho}{2}\ten  \right) \dest^1\left(-x_3 \rab(\beeta)\right)\sqrt{a}dx
 \\
 & \qquad - \intO \frac{\left( \theta \Lambda\right)^2}{\theta + \rho} \int_0^t e^{-k(t-s)}a^{\sigma \tau} \est^1(s) ds a^{\alpha \beta}\left(-x_3 \rab(\beeta)\right) \sqrt{a} dx
 \\
&\quad =-\frac{1}{2}\intO x_3 \a \est^1\rab(\beeta)\sqrt{a}dx - \frac{1}{2}\intO x_3 \b \dest^1 \rab(\beeta)\sqrt{a}dx 
   \\
   & \qquad+ \frac{1}{2}\int_0^te^{-k(t-s)}\intO x_3 \c \est^1(s)\rab(\beeta)\sqrt{a}dx ds 
   \\ & \quad 
 =\int_{\omega}p^{i}\eta_i\sqrt{a}dy,
 \end{align*}
 with $p^{i}$ defined in (\ref{def_p}) and
 where  $\a$ , $\b$ and $\c$ denote the contravariant components of the fourth order two-dimensional   tensors, defined in (\ref{tensor_a_bidimensional})--(\ref{tensor_c_bidimensional}). Now,recall that by the step $(iii)$ we  had that $-\d_3\eab^1=\rab(\bu)$ in $\WLO$ hence,
 \begin{align}\label{eab1_ups}
\eab^1=\Upsilon_{\alpha\beta}- x_3\rab(\bar{\bu}),
 \end{align}
  with $\Upsilon_{\alpha\beta}\in\WLo$ (independent of $x_3$). Therefore, from the previous equation we find that
 \begin{align} \nonumber
 &\frac{1}{2}\intO x_3^2 \a \rab(\bar{\bu})\rab(\beeta)\sqrt{a}dx + \frac{1}{2}\intO x_3^2 \b \rab(\dot{\bar{\bu}}) \rab(\beeta)\sqrt{a}dx 
    \\ \nonumber
    & \qquad- \frac{1}{2}\int_0^te^{-k(t-s)}\intO x_3^2 \c \rab(\bar{\bu}(s))\rab(\beeta)\sqrt{a}dx ds 
    \\ \label{ubarra}
     & \quad 
  =\int_{\omega}p^{i}\eta_i\sqrt{a}dy,
 \end{align}
  and since $\int_{-1}^1 x_3^2 dx_3=\frac{2}{3}$, we conclude that $\bar{\bu}$ verifies the equation of the Problem \ref{problema_ab}.
  
 \item {\em The strong convergences $\frac{1}{\var}\eij(\var)\to \eij^1$ in $\WLO$, hold. Moreover, since the limits $\eij^1$ are unique, then these convergences hold for the whole family $\left(\frac{1}{\var}\eij(\var)\right)_{\var>0}$}.
 
Indeed, we define
\begin{align}\nonumber
\Psi(\varepsilon)&:=\int_{\Omega}A^{ijkl}(\varepsilon)\left(\frac{1}{\var}\ekl(\varepsilon)-\ekl^1\right)\left(\frac{1}{\var}\eij(\varepsilon)-\eij^1\right)\sqrt{g(\varepsilon)}dx \\ \nonumber
& \qquad+\int_{\Omega}B^{ijkl}(\varepsilon)\left(\frac{1}{\var}\dekl(\varepsilon)-\dekl^1\right)\left(\frac{1}{\var}\eij(\varepsilon)-\eij^1\right)\sqrt{g(\varepsilon)}dx
\\ \nonumber
&\quad=\int_{\Omega} f^iu_i(\varepsilon)\sqrt{g(\varepsilon)}dx +\int_{\Gamma_1} h^i u_i(\var) \sqrt{a} d\Gamma 
\\ \nonumber
&\qquad- \int_{\Omega} A^{ijkl}(\varepsilon)\left(\frac{2}{\var}\ekl(\varepsilon)-\ekl^1\right)\eij^1\sqrt{g(\varepsilon)}dx
\\ \label{psi4}
&\qquad + \int_{\Omega} B^{ijkl}(\varepsilon)\left(\dekl^1\eij^1 - \frac{1}{\var}\frac{\d}{\d t}(\ekl(\var)\eij^1)\right)\sqrt{g(\varepsilon)}dx.
\end{align}
We have that,
 \begin{align}\nonumber
&\int_{\Omega}A^{ijkl}(\varepsilon)\left(\frac{1}{\var}\ekl(\varepsilon)-\ekl^1\right)\left(\frac{1}{\var}\eij(\varepsilon)-\eij^1\right)\sqrt{g(\varepsilon)}dx \\ \nonumber
& \qquad+\frac{1}{2}\frac{\d}{\d t}\int_{\Omega}B^{ijkl}(\varepsilon)\left(\frac{1}{\var}\ekl(\varepsilon)-\ekl^1\right)\left(\frac{1}{\var}\eij(\varepsilon)-\eij^1\right)\sqrt{g(\varepsilon)}dx ={\Psi}(\varepsilon), \ae.
 \end{align}
Integrating over the interval $[0,T]$, using (\ref{elipticidadB_eps}) and (\ref{condicion_inicial_def}) we find that
  \begin{align}\nonumber
\int_0^T\left(\int_{\Omega}A^{ijkl}(\varepsilon)\left(\frac{1}{\var}\ekl(\varepsilon)-\ekl^1\right)\left(\frac{1}{\var}\eij(\varepsilon)-\eij^1\right)\sqrt{g(\varepsilon)}dx\right) dt
  \quad\leq \int_{0}^{T}{\Psi}(\varepsilon)dt,
  \end{align}
Now, by (\ref{elipticidadA_eps}) and (\ref{g_acotado})
 \begin{align}\nonumber
  &{C}_e^{-1}g_0^{1/2}\sum_{i,j}\left|\frac{1}{\var}\eij(\varepsilon)-\eij^1\right|^2_{0,\Omega}
   \leq\int_{\Omega}A^{ijkl}(\varepsilon)\left(\frac{1}{\var}\ekl(\varepsilon)-\ekl^1\right)\left(\frac{1}{\var}\eij(\varepsilon)-\eij^1\right)\sqrt{g(\varepsilon)}dx.
  \end{align}
Therefore, together with the previous inequality leads to
  \begin{align}\label{ref1}
  &{C}_e^{-1}g_0^{1/2}\int_0^T\left(\sum_{i,j}\left|\frac{1}{\var}\eij(\varepsilon)-\eij^1\right|^2_{0,\Omega}\right)dt \leq \int_0^T \Psi(\varepsilon) dt.
  \end{align}
Let $\varepsilon\rightarrow 0$. Taking into account the weak convergences studied in $(i)$ and the asymptotic behaviour of the functions  $A^{ijkl}(\varepsilon),B^{ijkl}(\varepsilon)$ (see Theorem \ref{Th_comportamiento asintotico}) and $g(\varepsilon)$ (see Theorem \ref{Th_simbolos2D_3D}), we find that 
 \begin{align}\nonumber 
 \Psi:=&\lim_{\varepsilon\rightarrow 0}\Psi(\var)=\int_{\Omega} f^iu_i\sqrt{a}dx +\int_{\Gamma_1} h^i u_i \sqrt{a} d\Gamma
 \\ \label{psi2}
 & \quad -\int_{\Omega}A^{ijkl}(0)\ekl^1\eij^1\sqrt{a}dx - \int_{\Omega}B^{ijkl}(0)\dekl^1\eij^1\sqrt{a}dx,
 \end{align}
  $\ae.$ By the expressions of $A^{ijkl}(0)$ and $B^{ijkl}(0)$ (see Theorem \ref{Th_comportamiento asintotico}) and that $\eatres^1=0$ (see step $(iii)$), we have that
  \begin{align}\nonumber
   &\intO A^{ijkl}(0)\ekl^1\eij^1\sqrt{a} dx + \intO B^{ijkl}(0)\dekl^1\eij^1 \sqrt{a} dx \\ \nonumber
   & \quad = \intO \left( \lambda a^{\alpha \beta} a^{\sigma \tau} + \mu (a^{\alpha \sigma}a^{\beta \tau} + a^{\alpha \tau}a^{\beta\sigma})  \right) \est^1 \eab^1 \sqrt{a} dx 
    + \intO \lambda a^{\alpha\beta}\edtres^1 \eab^1\sqrt{a} dx
    \\ \nonumber
    & \qquad  + \intO\left(\lambda a^{\sigma\tau} \est^1 + (\lambda + 2\mu)\edtres^1 \right)\edtres^1 \sqrt{a} dx
   \\ \nonumber
   & \qquad + \intO \left( \theta a^{\alpha \beta} a^{\sigma \tau} + \frac{\rho}{2} (a^{\alpha \sigma}a^{\beta \tau} + a^{\alpha \tau}a^{\beta\sigma})  \right) \dest^1\eab^1 \sqrt{a} dx
   + \intO \theta a^{\alpha\beta}\dedtres^1\eab^1\sqrt{a} dx 
     \\ \nonumber
       & \qquad  + \intO\left(\theta a^{\sigma\tau} \dest^1 + (\theta + \rho)\dedtres^1 \right)\edtres^1 \sqrt{a} dx, \ae,
   \end{align}
 which using the expressions (\ref{edtres1})--(\ref{dedtres1}) studied in $(iii)$ can be written as 
   \begin{align*}
   &\intO \left( \left(\lambda - \frac{\theta}{\theta + \rho} \left( \theta \Lambda + \lambda \right)\right)a^{\alpha\beta}a^{\sigma \tau} + \mu\ten  \right) \est^1\eab^1\sqrt{a}dx
   \\
   &\qquad + \intO \left( \frac{\theta \rho }{\theta + \rho} a^{\alpha \beta} a^{\sigma \tau}  + \frac{\rho}{2}\ten  \right) \dest^1\eab^1\sqrt{a}dx
   \\
   & \qquad - \intO \frac{\left( \theta \Lambda\right)^2}{\theta + \rho} \int_0^t e^{-k(t-s)}a^{\sigma \tau} \est^1(s) ds a^{\alpha \beta}\eab^1(t) \sqrt{a} dx,
   \end{align*}
thus,
   \begin{align} \nonumber
   \Psi&=\int_{\omega} p^i\bar{u}_i\sqrt{a}dy -\frac{1}{2}\intO \a \est^1\eab^1\sqrt{a}dx- \frac{1}{2}\intO \b \dest^1\eab^1\sqrt{a}dx 
    \\ \label{psi1}
    & \qquad+ \frac{1}{2}\int_0^te^{-k(t-s)}\intO \c\est^1(s)\eab^1(t)\sqrt{a}dx ds, \ae,
   \end{align}
 where  $\a$ , $\b$ and $\c$ denote the contravariant components of the fourth order two-dimensional  tensors, defined  in (\ref{tensor_a_bidimensional})--(\ref{tensor_c_bidimensional}). Now, recall (\ref{eab1_ups}), hence
 \begin{align}\nonumber
 &\frac{1}{2}\intO \a \est^1\eab^1\sqrt{a}dx+ \frac{1}{2}\intO \b \dest^1\eab^1\sqrt{a}dx 
     \\ \nonumber
     & \qquad- \frac{1}{2}\int_0^te^{-k(t-s)}\intO \c\est^1(s)\eab^1(t)\sqrt{a}dx ds, 
     \\ \nonumber
     &\quad=\into \a \Upsilon_{\sigma\tau}\Upsilon_{\alpha\beta}\sqrt{a}dx+ \frac{1}{3}\into \a \rst(\bar{\bu})\rab(\bar{\bu})\sqrt{a}dx+ \into \b \dot{\Upsilon}_{\sigma\tau}\Upsilon_{\alpha\beta}\sqrt{a}dx 
          \\ \nonumber
          & \qquad+  \frac{1}{3}\into \b \rst(\dot{\bar{\bu}})\rab(\bar{\bu})\sqrt{a}dx- \int_0^te^{-k(t-s)}\into \c\Upsilon_{\sigma\tau}(s)\Upsilon_{\alpha\beta}(t)\sqrt{a}dx ds
          \\\label{rel_mag2}
          & \qquad- \frac{1}{3}\int_0^te^{-k(t-s)}\into \c\rst(\bar{\bu}(s))\rab(\bar{\bu}(t))\sqrt{a}dx ds , \aes.
 \end{align}
    By the step $(iv)$, $\bar{\bu}$ verifies the equation of the Problem \ref{problema_ab}, hence from (\ref{psi1}) we find that
    \begin{align} \nonumber
      \Psi&= -\into \a \Upsilon_{\sigma\tau}\Upsilon_{\alpha\beta}\sqrt{a}dy -\into \b \dot{\Upsilon}_{\sigma\tau}\Upsilon_{\alpha\beta}\sqrt{a}dy
      \\ \label{psi3}
      & \quad + \int_0^te^{-k(t-s)}\into \c\Upsilon_{\sigma\tau}(s)\Upsilon_{\alpha\beta}(t)\sqrt{a}dx ds , \aes.
    \end{align}
    By (\ref{ref1}) we have that $\int_0^T\Psi(\var) dt \geq 0$ for all $\var>0$, then  $\int_0^T\Psi dt \geq 0$. In order to prove the opposite inequality, let us define $\mathbf{\Upsilon}:=\left(\Upsilon_{ij}\right)\in\mathbb{S}^3$, such that $\Upsilon_{\alpha\beta}$  denote the elements introduced in (\ref{eab1_ups}), $\Upsilon_{\alpha 3}:=0$ and where  $\Upsilon_{33}\in\WLo$ are defined by the expression
     \begin{align} \label{ups33}
     \Upsilon_{33}(t)= - \frac{\theta}{\theta + \rho} \left( a^{\alpha \beta }\Upsilon_{\alpha\beta}(t) + \Lambda\int_0^te^{-k(t-s)}a^{\alpha\beta}\Upsilon_{\alpha\beta}(s) ds \right),
      \end{align}
      in $\Omega$ , $\forallt$, and where $\Lambda$ and $k$ are defined  in (\ref{Lambda}) and (\ref{k}), respectively. As a consequence, we have that
      \begin{align}\label{dups33} 
      \dot{\Upsilon}_{33}(t)= - \frac{\lambda}{\theta + \rho} a^{\alpha \beta} \Upsilon_{\alpha\beta}(t)- \frac{\lambda + 2\mu}{\theta + \rho}\Upsilon_{33}(t) - \frac{\theta}{\theta + \rho}  a^{\alpha\beta}\dot{\Upsilon}_{\alpha\beta}(t).
      \end{align}
      in $\Omega$ , $\ae.$ Taking this into account and
      using the expressions of $A^{ijkl}(0)$ and $B^{ijkl}(0)$ (see Theorem \ref{Th_comportamiento asintotico}), we have that
        \begin{align}\nonumber
         &\intO A^{ijkl}(0)\Upskl \Upsij\sqrt{a} dx + \intO B^{ijkl}(0)\dUpskl \Upsij \sqrt{a} dx \\ \nonumber
         & \quad = \intO \left( \lambda a^{\alpha \beta} a^{\sigma \tau} + \mu (a^{\alpha \sigma}a^{\beta \tau} + a^{\alpha \tau}a^{\beta\sigma})  \right)\Upsst \Upsab \sqrt{a} dx 
          + \intO \lambda a^{\alpha\beta}\Upsdtres \Upsab\sqrt{a} dx
          \\ \nonumber
          & \qquad  + \intO\left(\lambda a^{\sigma\tau}\Upsst + (\lambda + 2\mu)\Upsdtres \right)\Upsdtres \sqrt{a} dx
         \\ \nonumber
         & \qquad + \intO \left( \theta a^{\alpha \beta} a^{\sigma \tau} + \frac{\rho}{2} (a^{\alpha \sigma}a^{\beta \tau} + a^{\alpha \tau}a^{\beta\sigma})  \right) \dUpsst \Upsab \sqrt{a} dx
         + \intO \theta a^{\alpha\beta} \dUpsdtres \Upsab yy \sqrt{a} dx 
           \\ \label{calc_1}
             & \qquad  + \intO\left(\theta a^{\sigma\tau}\dUpsst + (\theta + \rho)\dUpsdtres \right)\Upsdtres \sqrt{a} dx, \ae,
         \end{align}
       which using the expressions (\ref{ups33})--(\ref{dups33})  can be written as 
         \begin{align}\nonumber
         &\intO \left( \left(\lambda - \frac{\theta}{\theta + \rho} \left( \theta \Lambda + \lambda \right)\right)a^{\alpha\beta}a^{\sigma \tau} + \mu\ten  \right) \Upsst \Upsab \sqrt{a}dx
         \\\nonumber
         &\qquad + \intO \left( \frac{\theta \rho }{\theta + \rho} a^{\alpha \beta} a^{\sigma \tau}  + \frac{\rho}{2}\ten  \right) \dUpsst \Upsab \sqrt{a}dx
         \\\label{calc_2}
         & \qquad - \intO \frac{\left( \theta \Lambda\right)^2}{\theta + \rho} \int_0^t e^{-k(t-s)}a^{\sigma \tau}\Upsst (s) ds a^{\alpha \beta}\Upsab(t) \sqrt{a} dx,
         \end{align}
     thus,
   \begin{align*}
    &\intO A^{ijkl}(0)\Upskl \Upsij\sqrt{a} dx + \intO B^{ijkl}(0)\dUpskl \Upsij \sqrt{a} dx 
    \\ \nonumber
    &= \into \a \Upsilon_{\sigma\tau}\Upsilon_{\alpha\beta}\sqrt{a}dy + \into \b \dot{\Upsilon}_{\sigma\tau}\Upsilon_{\alpha\beta}\sqrt{a}dy
          \\
          & \quad - \int_0^te^{-k(t-s)}\into \c\Upsilon_{\sigma\tau}(s)\Upsilon_{\alpha\beta}(t)\sqrt{a}dy ds , \aes.
   \end{align*}
   Hence, from (\ref{psi3}) we find that,
   \begin{align}\label{psi5}
   \Psi=-\intO A^{ijkl}(0)\Upskl \Upsij\sqrt{a} dx - \intO B^{ijkl}(0)\dUpskl \Upsij \sqrt{a} dx.
   \end{align}
  Now, since the functions $\Upsilon_{\alpha\beta}$ are independent of $x_3$ and $\eij^1(0,\cdot)=0$ by (\ref{condicion_inicial_def}) and the weak convergences from $(i)$ (applying a result that can be found in Lemma 2.55, \cite{MiOchSo}), then $\rab(\bar{\bu}(0,\cdot))$ must be zero and, as a consequence, $\Upsilon_{\alpha\beta}(0,\cdot)=0.$ Moreover, by $(\ref{ups33})$ we have that $\Upsdtres(0,\cdot)=0,$ as well. Therefore, if we integrate (\ref{psi5})  over $[0,T]$ and take into account the initial conditions $\Upsilon_{ij}(0,\cdot)=0$ and the  ellipticity of tensors $\left(A^{ijkl}(0)\right)$ and $\left(B^{ijkl}(0)\right)$ (see (\ref{elipticidadA_eps0})), we find that $\int_0^T\Psi(s) ds\leq0$ so we conclude that
    \begin{align*}
    \int_0^T\Psi(s) ds=0.
    \end{align*}
      Therefore, by (\ref{ref1}) the strong convergences $\frac{1}{\var}\eij(\var)\to \eij^1$ in $\LLO$, when $\var\to0$, hold. On the other hand, if we define
     \begin{align}\nonumber
     \tilde{\Psi}(\varepsilon)&:=\int_{\Omega}A^{ijkl}(\varepsilon)\left(\frac{1}{\var}\ekl(\varepsilon)-\ekl^1\right)\left(\frac{1}{\var}\deij(\varepsilon)-\deij^1\right)\sqrt{g(\varepsilon)}dx \\ \nonumber
     & \qquad+\int_{\Omega}B^{ijkl}(\varepsilon)\left(\frac{1}{\var}\dekl(\varepsilon)-\dekl^1\right)\left(\frac{1}{\var}\deij(\varepsilon)-\deij^1\right)\sqrt{g(\varepsilon)}dx
     \\ \nonumber
     &\quad=\int_{\Omega} \fb^i\dot{u}_i(\varepsilon)\sqrt{g(\varepsilon)}dx +\int_{\Gamma_1} h^i \dot{u}_i(\var) \sqrt{a} d\Gamma
     \\ \nonumber
     & \qquad +\int_{\Omega} A^{ijkl}(\varepsilon)\left(\ekl^1\deij^1 - \frac{1}{\var}\frac{\d}{\d t}(\ekl(\var)\eij^1)\right)\sqrt{g(\varepsilon)}dx
     \\ \nonumber
     &\qquad -\int_{\Omega} B^{ijkl}(\varepsilon)\left(\frac{2}{\var}\dekl(\varepsilon)-\dekl^1\right)\deij^1\sqrt{g(\varepsilon)} dx.
     \end{align} 
     We have that,
      \begin{align}\nonumber
      \frac{1}{2}&\frac{\d}{\d t}\int_{\Omega}A^{ijkl}(\varepsilon)\left(\frac{1}{\var}\ekl(\varepsilon)-\ekl^1\right)\left(\frac{1}{\var}\eij(\varepsilon)-\eij^1\right)\sqrt{g(\varepsilon)}dx
      \\ \nonumber
      & +\int_{\Omega}B^{ijkl}(\varepsilon)\left(\frac{1}{\var}\dekl(\varepsilon)-\dekl^1\right)\left(\frac{1}{\var}\deij(\varepsilon)-\deij^1\right)\sqrt{g(\varepsilon)}dx =\tilde{\Psi}(\varepsilon), \ae.
      \end{align}
     Integrating over the interval $[0,T]$, using (\ref{elipticidadA_eps}) and (\ref{condicion_inicial_def}) we find that
       \begin{align}\nonumber
     \int_0^T\left(\int_{\Omega}B^{ijkl}(\varepsilon)\left(\frac{1}{\var}\dekl(\varepsilon)-\dekl^1\right)\left(\frac{1}{\var}\deij(\varepsilon)-\deij^1\right)\sqrt{g(\varepsilon)}dx\right) dt
       \quad\leq \int_{0}^{T}\tilde{\Psi}(\varepsilon)dt,
       \end{align}
     Now, by (\ref{elipticidadB_eps}) and (\ref{g_acotado})
      \begin{align}\nonumber
       &{C}_v^{-1}g_0^{1/2}\sum_{i,j}\left|\frac{1}{\var}\deij(\varepsilon)-\deij^1\right|^2_{0,\Omega}
        \leq\int_{\Omega}B^{ijkl}(\varepsilon)\left(\frac{1}{\var}\dekl(\varepsilon)-\dekl^1\right)\left(\frac{1}{\var}\deij(\varepsilon)-\deij^1\right)\sqrt{g(\varepsilon)}dx
       \end{align}
     Therefore, together with the previous inequality leads to
      \begin{align}\label{ref3}
       &{C}_v^{-1}g_0^{1/2}\int_0^T\left(\sum_{i,j}\left|\frac{1}{\var}\deij(\varepsilon)(t)-\deij^1(t)\right|^2_{0,\Omega}\right)dt \leq \int_0^T \tilde{\Psi}(\varepsilon) dt,
       \end{align}
     which is similar with (\ref{ref1}). Therefore, using analogous arguments as before, we find that
      \begin{align}
        \tilde{\Psi}&:=\lim_{\var\rightarrow 0} \tilde{\Psi}(\var)=\int_{\Omega} f^i\dot{u}_i\sqrt{a}dx +\int_{\Gamma_1} h^i \dot{u}_i \sqrt{a} d\Gamma
        \\
        & \qquad -\frac{1}{2}\intO \a \est^1\deab^1\sqrt{a}dx- \frac{1}{2}\intO \b \dest^1\deab^1\sqrt{a}dx 
         \\ \label{psi7}
         & \qquad+ \frac{1}{2}\int_0^te^{-k(t-s)}\intO \c\est^1(s)\deab^1(t)\sqrt{a}dx ds, \ae.
        \end{align}
      Now, following similar arguments as before, taking into account the time derivative of the relation (\ref{eab1_ups}) and taking $\beeta=\dot{\bar{\bu}}$  in (\ref{ubarra}), we find that 
       \begin{align} \nonumber
            \tilde{\Psi}&= -\into \a \Upsilon_{\sigma\tau}\dUpsab\sqrt{a}dy -\into \b \dUpsst\dUpsab\sqrt{a}dy
            \\ \label{psi6}
            & \quad + \int_0^te^{-k(t-s)}\into \c\Upsilon_{\sigma\tau}(s)\dUpsab(t)\sqrt{a}dy ds , \aes.
          \end{align}
      Again, taking into account the definition of the tensor $\mathbf{\Upsilon}:=(\Upsij)\in\mathbb{S}$ where $\Upsilon_{\alpha 3}=0$ and $\Upsilon_{3 3}$ given by (\ref{ups33}) and repeating analogous calculations as in (\ref{calc_1})--(\ref{calc_2}) we find that
       \begin{align*}
        &\intO A^{ijkl}(0)\Upskl \dUpsij\sqrt{a} dx + \intO B^{ijkl}(0)\dUpskl \dUpsij \sqrt{a} dx 
        \\ \nonumber
        &= \into \a \Upsilon_{\sigma\tau}\dUpsab\sqrt{a}dy + \into \b \dot{\Upsilon}_{\sigma\tau}\dUpsab\sqrt{a}dy
              \\
              & \quad - \int_0^te^{-k(t-s)}\into \c\Upsilon_{\sigma\tau}(s)\dUpsab(t)\sqrt{a}dy ds , \aes.
       \end{align*}
       Hence, from (\ref{psi7}) we find that,
       \begin{align}\label{psi8}
       \tilde{\Psi}=-\intO A^{ijkl}(0)\Upskl \dUpsij\sqrt{a} dx - \intO B^{ijkl}(0)\dUpskl \dUpsij \sqrt{a} dx.
       \end{align}
      Therefore, if we integrate (\ref{psi8})  over $[0,T]$ and take into account the initial conditions $\Upsilon_{ij}(0,\cdot)=0$ and the  ellipticity of tensors $\left(A^{ijkl}(0)\right)$ and $\left(B^{ijkl}(0)\right)$ (see (\ref{elipticidadA_eps0})), we find that $\int_0^T\tilde{\Psi}(s) ds\leq0$ so we conclude that
        \begin{align*}
        \int_0^T\tilde{\Psi}(s) ds=0.
        \end{align*}
          Therefore, by (\ref{ref3}) the strong convergences $\frac{1}{\var}\deij(\var)\to \deij^1$ in $\LLO$, when $\var\to0$, hold. Thus, we conclude that  $\frac{1}{\var}\eij(\var)\to \eij^1$ in $\WLO$, when $\var\to0$.  In particular, this convergence implies that $\frac{1}{\var}\eij(\var)\to \eij^1$ in $L^2(\Omega)$ for all $t\in[0,T].$ Now, taking the limit when $\var\to 0$ in the first equality of (\ref{psi4}) we find that $\Psi=0$. Hence, by using an existence and uniqueness result that can be found in Theorem 4.10 in \cite{intro2}, we can ensure that (\ref{psi3}) (with $\Psi=0$) implies that $\Upsab(t)=0$ for all $t\in[0,T]$. This also leads to,
          \begin{align*}
          \eab^1=-x_3\rab(\bar{\bu}),
          \end{align*}
      this is, the functions $\eab^1$ are unequivocally determined, since $\bar{\bu}$ is unique (it is the unique solution of the Problem \ref{problema_ab}, by Theorem \ref{Th_exist_unic_bid_cero}). As a consequence, by the relations found in $(iii)$ we obtain that the functions $e_{i||3}$ are also unique.

   \item {\em The strong convergence $\bu(\var)\to\bu$ in $\WHOt$ holds}.
   
   Since $\frac{1}{\var}\eij(\varepsilon) \rightarrow\eij^1 $ in $\WLO$ by the step $(v)$ we have that
   \begin{align*}
   \frac{1}{\var}\d_3\eij(\varepsilon)\to\d_3\eij^1 \in \WHMO. 
   \end{align*}
   Then, by the Theorem \ref{Th_521} we have that
   \begin{align*}
   \left(\rab(\bu(\var)) +  \frac{1}{\var}\d_3\eij(\varepsilon)\right) \to 0 \en \WHMO.
   \end{align*} 
    Hence,
    \begin{align*}
    \rab(\bu(\var)) \to -\d_3\eij^1 \en \WHMO,
    \end{align*}
    that is, each family $\left(\rab(\bu(\var))\right)_{\var>0}$ converges in $\WHMO$,  then the conclusion follows by $(e)$ in Theorem \ref{Th_522}. Hence, the proof of the step $(a)$ of this theorem is met.
  \end{enumerate}
Therefore, the proof of the theorem is complete.
   \end{proof}
It remains to be proved an analogous result to the previous theorem but in terms of de-scaled unknowns. The convergence $\bu(\var) \to \bu$ in $\WHOt$ the Theorem \ref{Th_convergencia}, the scaling proposed in Section \ref{seccion_dominio_ind}, the de-scalings $\xi_i^\var:=\xi_i$ for each $\var>0$ and  the Theorem \ref{Th_medias} together lead to the following convergences:
\begin{align*}
\frac{1}{2\var}\int^\var_{-\var} u_i^\var dx_3^{\var} \to \xi_i  \ \textrm{in} \ \WHo.
\end{align*} 
Furthermore, we can prove the following theorem regarding the convergences of the averages of the tangential and normal components of the three-dimensional displacement vector field:
\begin{theorem}
 Assume that $\btheta\in\mathcal{C}^3(\bar{\omega};\mathbb{R}^3)$. Consider a family of viscoelastic flexural shells with thickness $2\var$ approaching zero and with each having the same  middle surface $S=\btheta(\bar{\omega})$, and let the assumptions on the data be as in Theorem \ref{Th_exist_unic_bid_cero}.
 
 Let $\bu^\var=(u_i^\var)\in H^1(0,T, V(\Omega^\var))$ and $\bxi^\var=(\xi_i^\var)\in H^{1}(0,T, V_F(\omega))$ respectively denote for each $\var>0$ the solutions to the three-dimensional and two-dimensional Problems \ref{problema_eps} and \ref{problema_ab_eps}. Moreover, let $\bxi=(\xi_i)\in H^{1}(0,T, V_F(\omega))$ denote the solution to the Problem \ref{problema_ab}. Then we have that
 \begin{align*}
 \xi_\alpha^\var=\xi_\alpha \ \textrm{and thus} \ \xi_\alpha^\var\ba^\alpha=\xi_\alpha\ba^\alpha \ \textrm{in} \ H^{1}(0, T, H^1(\omega)),
  \forall \var>0,
  \\
  \frac{1}{2\var}\int^\var_{-\var} u_\alpha^\var \bg^{\alpha,\var} dx_3^\var \to \xi_\alpha \ba^\alpha \ \textrm{in} \ H^{1}(0, T, H^1(\omega)) \ \textrm{as} \ \var \to 0,
 \end{align*}
 and
  \begin{align*}
  \xi_3^\var=\xi_3 \ \textrm{and thus} \ \xi_3^\var\ba^3=\xi_3\ba^3 \ \textrm{in} \ H^{1}(0, T; H^2(\omega)),
   \forall \var>0,
   \\
   \frac{1}{2\var}\int^\var_{-\var} u_3^\var \bg^{3,\var} dx_3^\var \to \xi_3 \ba^3 \ \textrm{in} \ H^{1}(0, T; H^1(\omega)) \ \textrm{as} \ \var \to 0.
  \end{align*}
\end{theorem}
\begin{proof}
Since $\btheta\in\mathcal{C}^3(\bar{\omega};\mathbb{R}^3)$ the vector fields $\bg^\alpha(\var): \bar{\Omega}\longrightarrow \mathbb{R}^3$ defined by $\bg^\alpha(\var):=\bg^{\alpha,\var}(\bx^\var)$ for all $\bx^\var=\pi(\bx)\in \bar{\Omega}^\var$ are such that $\bg^\alpha(\var)- \ba^\alpha= O(\var)$, where the fields $\ba^\alpha$ have been identified with vector fields defined over the whole set $\bar{\Omega}$. Now we have that,
\begin{align*}
&\frac{1}{2 \var} \int_{-\var}^\var u_\alpha^\var\bg^{\alpha, \var} dx_3^\var - \xi_\alpha^\var \ba^\alpha= \frac{1}{2}\int^{1}_{-1} u_\alpha(\var)\bg^\alpha(\var)dx_3 - \xi_\alpha\ba^\alpha
\\ 
&\quad= \frac{1}{2}\int^1_{-1}u_\alpha(\var)(\bg^\alpha(\var)-\ba^\alpha) dx_3 - (\overline{u_\alpha(\var)} - \xi_\alpha)\ba^\alpha.
\end{align*}
On one hand, since $u_\alpha(\var)\to u_\alpha$ in $H^{1}(0,T;H^1(\Omega))$ and $\bg^\alpha(\var)\to \ba^\alpha$ in $\mathcal{C}^1(\bar{\Omega})$ imply that
\begin{align*}
u_\alpha(\var)(\bg^\alpha(\var)-\ba^\alpha)\to 0 \ \textrm{in} \ H^{1}(0,T;H^1(\Omega)),
\end{align*}
hence, applying Theorem \ref{Th_medias} $(b)$ we have that
\begin{align*}
\frac{1}{2}\int_{-1}^{1} u_\alpha(\var)(\bg^\alpha(\var)-\ba^\alpha) dx_3 \to 0 \ \textrm{in} \ H^{1}(0,T;H^1(\omega)),
\end{align*}
and by using the same argument we have that $(\overline{u_\alpha(\var)} - \xi_\alpha)\ba^\alpha \to 0$ in $H^{1}(0,T;H^1(\omega))$. For the normal components we have that $\bg^{3,\var}=\ba^3$, then
\begin{align*}
&\frac{1}{2 \var} \int_{-\var}^\var u_3^\var\bg^{3, \var} dx_3^\var - \xi_3^\var \ba^3=  (\overline{u_3(\var)} - \xi_3)\ba^3,
\end{align*}
hence applying Theorem \ref{Th_medias} $(a)$ we have that $(\overline{u_3(\var)} - \xi_3)\ba^3 \to 0$ in $H^{1}(0,T;H^1(\omega))$.
\end{proof}

\begin{remark}
The fields $\tilde{\bxi}_T^\var, \tilde{\bxi}_N^\var: [0,T]\times \bar{\omega}\longrightarrow\mathbb{R}^3$ defined by $\tilde{\bxi}_T^\var:=\xi_\alpha^\var\ba^\alpha$ and $\tilde{\bxi}_N^\var:=\xi_3^\var\ba^3$, are known as the limit tangential and normal displacement fields, respectively, of the middle surface $S$ of the shell. If we denote the limit displacement field of $S$ by $\tilde{\bxi}^\var:=\xi_i\ba^i$ then $\tilde{\bxi}^\var=\tilde{\bxi}_T^\var+\tilde{\bxi}_N^\var.$
\end{remark}

\section{Conclusions} \label{conclusiones}

We have found and mathematically justified a model for  viscoelastic shells in the particular case of the so-called flexural shells. To this end we used the insight provided by the  asymptotic expansion method (presented in our previous work \cite{intro2}) and we have justified this approach by obtaining convergence theorems.

The main novelty that this model presented  is a long-term memory, represented by an integral on the time variable, more specifically
\begin{align*}
 M(t,\beeta)=\int_0^te^{-k(t-s)}\into \c \rst(\bxi(s))\rab(\beeta)\sqrt{a}dyds , 
\end{align*}
for all $\beeta\in V_F(\omega)$. An analogous behaviour has been found in beam models for the bending-stretching of viscoelastic rods \cite{AV}, obtained by using asymptotic methods as well.  Also, this kind of viscoelasticity has been described in \cite{DL,Pipkin}, for instance.

As the viscoelastic case differs from the elastic case on time dependent constitutive law and external forces, we must consider the possibility that these models and the convergence result generalize the elastic case (studied in \cite{Ciarlet4b}). However, analogously to the asymptotic analysis made in \cite{intro2}, the reader can easily check, when the ordinary differential equation (\ref{ecuacion_casuistica}) was presented, we  had to use assumptions that make it impossible to include the elastic case.  Hence, the viscoelastic and elastic problems must be treated separately in order to reach reasonable and justified conclusions.

In this paper we have presented the convergence results concerning the models for the so-called  viscoelastic flexural shells where we assumed that $V_F(\omega)\neq \{\bcero\}$. Concerning the remaining cases where the space $V_F(\omega)$ reduces only to the zero element, in \cite{eliptico} and \cite{generalizada} we present the corresponding mathematical justifications of the models known as viscoelastic membrane shell problems. In the first one \cite{eliptico,eliptico2},   we consider a family of shells where each one as the same elliptic middle surface $S$ and the boundary condition is considered in the whole lateral face of the shell. This set of problems will be known as the viscoelastic elliptic membrane shells. In the later one \cite{generalizada}, we shall consider the remaining cases where one of these hypothesis does not verify but still $V_F(\omega)= \{\bcero\}$. This set of problems will be known as the viscoelastic generalized membrane shells.

\section*{Acknowledgements}
{\footnotesize \noindent This research was partially supported by Ministerio de Econom\'ia y Competitividad of Spain, under the grant MTM2016-78718-P, with the participation of FEDER.}
\section*{References}

\end{document}